\documentclass[letter,11pt]{amsart}  
\usepackage[T1]{fontenc}
\usepackage[usenames,dvipsnames]{color}
\usepackage{amssymb,amsmath,amsthm,mathrsfs}
\usepackage{graphicx}
\usepackage{mathpazo}
\usepackage{stmaryrd} 
\usepackage{hhline} 
\usepackage{lscape}

\usepackage[all]{xy}

\usepackage[colorlinks=true,linkcolor=Cerulean,pagebackref,citecolor=Mahogany]{hyperref}
\usepackage{enumerate}

 \usepackage{tikz}
\usetikzlibrary{arrows,decorations.pathmorphing,decorations.pathreplacing,positioning,shapes.geometric,shapes.misc,decorations.markings,decorations.fractals,calc,patterns,backgrounds}
\usepackage{tikz-cd}

\tikzset{>=stealth',
     cvertex/.style={circle,draw=black,inner sep=1pt,outer sep=3pt},
     vertex/.style={circle,fill=black,inner sep=1pt,outer sep=3pt},
     star/.style={circle,fill=yellow,inner sep=0.75pt,outer sep=0.75pt},
     tvertex/.style={inner sep=1pt,font=\criptsize},
     gap/.style={inner sep=0.5pt,fill=white}}

\newcommand{\arrowrl}[3][20]
{
\hspace{-5pt}
\begin{tikzpicture}
\node (A) at (0,0) {};
\node (B) at (1,0) {};
\draw[->] ($(A)+(0,0.2)$) -- node [above] {$\scriptstyle f^*$} ($(B)+(0,0.2)$);
\draw [->] ($(B)+(0,0.2)$) -- node [below] {$\scriptstyle f_*$} ($(A)+(0,0.2)$);
\end{tikzpicture}
\hspace{-5pt}
}
\newcommand{\adj}[2][20]{\arrowrl}

\newcommand*\atwo{ \tikz[baseline=(char.base),->,>=stealth',shorten >=1pt,auto,node distance=1.2cm, thick,main node/.style={circle,draw}]{
		\node[shape=circle,draw,inner sep=2pt] (char) {1};
		\node[shape=circle,draw,inner sep=2pt,right of=char] (2) {2};
		\path
		(char)	edge (2);}}

\newcommand*\aTwo{ \tikz[baseline=(char.base),->,>=stealth',shorten >=1pt,auto,node distance=1.2cm, thick,main node/.style={circle,draw}]{
		\node[shape=circle,draw,inner sep=2pt] (char) {1};
		\node[shape=circle,draw,inner sep=2pt,right of=char] (2) {2};
		\path
		(2)	edge (char);}}

\newcommand{\hoehe}{\textcolor{white}{\bigg|}}

\newcommand{\RR}{\ensuremath{\mathbb{R}}}
\newcommand{\ZZ}{\ensuremath{\mathbb{Z}}}
\newcommand{\CC}{\ensuremath{\mathbb{C}}}
\newcommand{\KK}{\ensuremath{K}} 
\newcommand{\QQ}{\ensuremath{\mathbb{Q}}}

\newcommand{\PP}{\ensuremath{\mathbb{P}}}
\newcommand{\TT}{\ensuremath{\mathbb{T}}}

\DeclareMathOperator{\Hom}{Hom}
\DeclareMathOperator{\Ext}{Ext}
\DeclareMathOperator{\End}{End}
\DeclareMathOperator{\add}{add}

\newcommand{\cc}{\mathbb{X}}

\newcommand{\mc}[1]{\ensuremath{\mathcal{#1}}}   

\newcommand{\cA}{\mathcal{A}}
\newcommand{\cF}{\mathbb{F}}

\newcommand{\cQ}{Q}

\newcommand{\cX}{\mathcal{X}}

\newcommand{\p}{\boldsymbol{p}} 
\newcommand{\s}{\boldsymbol{s}} 
\renewcommand{\t}{\boldsymbol{t}} 
\newcommand{\x}{\boldsymbol{x}} 
\newcommand{\y}{\boldsymbol{y}}

\newcommand{\CM}{\operatorname{{MCM}}}

\newcommand{\mmod}[1]{\operatorname{mod}(#1)}

\newcommand{\rot}[1]{{\color{red} #1}}
\newcommand{\grau}[1]{{\color{lightgray} #1}}

\theoremstyle{theorem}
\newtheorem{Thm}{Theorem}[section]         
\newtheorem{Lem}[Thm]{Lemma}

\newtheorem{Prop}[Thm]{Proposition}   

\newtheorem{Qu}[Thm]{Question}

\theoremstyle{remark}
\newtheorem{Rmk}[Thm]{Remark}
\newtheorem{Ex}[Thm]{Example}

\theoremstyle{definition}
\newtheorem{defi}[Thm]{Definition} 

\newtheorem{fact}[Thm]{Facts}

\title{Frieze patterns in representation theory}

\author{Eleonore Faber}

\address{Institut f\"ur Mathematik und Wissenschaftliches Rechnen,
Universit\"at Graz,
Heinrichstr.~36,
A-8010 Graz, Austria and School of Mathematics, University of Leeds, LS2 9JT Leeds, UK}

\email{eleonore.faber@uni-graz.at, e.m.faber@leeds.ac.uk}

\date{\today}
\thanks{ Work of E.F.~was supported by EPSRC grant EP/W007509/1 and by NAWI Graz.
 } 
 
\begin{document}


\maketitle

\begin{abstract}
Friezes patterns are infinite arrays of numbers, in which every four neighbouring vertices arranged in a diamond satisfy the same arithmetic rule. Introduced in the late 1960s by Coxeter, and further studied by Conway and Coxeter in their remarkable papers from $1973$, this topic has been nearly forgotten for over thirty years. But since the discovery of connections to cluster algebras and categories of type $A$, interest in friezes has exploded, several generalizations have been studied, and links to geometry and combinatorics have been explored.

In this article we will review  some of the most striking results connecting the purely combinatorially defined friezes with triangulations of polygons, Grassmannian cluster  algebras and (Grassmannian) cluster categories. Then we will focus  on Grassmannian cluster categories and some recent  results linking them to friezes.

\end{abstract}


\section{Introduction}

The goal of this article is to survey results on frieze patterns\footnote{Sometimes these are also denoted as \emph{friezes}. We will use frieze patterns and friezes interchangeably.} from the point of view of representation theory, based on my talk from the International Conference on Representations of Algebras (ICRA) 2024 in Shanghai. The main idea that I wanted to get across is that frieze patterns appear quite naturally in connection with (Grassmannian) cluster categories and that this yields interesting interactions with other fields, such as algebraic geometry and commutative algebra. \\
While  we will review some of the ``classical'' correspondence between triangulations of polygons, cluster algebras and categories of type $A$ and Conway--Coxeter frieze patterns, the focus is on Grassmannian cluster categories. In particular, we show how to obtain the Conway--Coxeter friezes from them  (including the boundary rows), and then survey recent results on the construction of tame integral $SL_k$-friezes and infinite friezes from Grassmannian cluster categories of infinite rank. 

What is in this survey: In Section \ref{Sec:friezes} we will introduce the main object of our studies, mathematical frieze patterns. Therefore, we roughly follow the historical trajectory, starting with Coxeter's and Conway--Coxeter's seminal work on combinatorics of frieze patterns from the 1970s. The topic then long lay dormant, until the introduction of cluster algebras and connections to cluster categories in the 2000s. Since then research on friezes has exploded. \\
We start  by describing some types of frieze patterns purely combinatorially and some of their properties that will be useful from a more representation theoretic point of view later: Conway--Coxeter friezes (see Section \ref{Sub:CCfrieze}), and their generalizations (see Section \ref{Sub:general}), i.e., tame $SL_k$-friezes, friezes with coefficients, and infinite periodic friezes. And we also give an outlook on friezes obtained by varying the defining diamond rule. 

In Section  \ref{Sec:cluster-algs-cats}, we provide the basics on cluster algebras and categories leading to the cluster character, which is the link to frieze patterns. Let us stress here that we focus on the category side of things\footnote{as customary for a talk at ICRA} and less on the algebra properties. So we define cluster algebras only very generally (although explaining the example of the homogeneous coordinate ring of Grassmannians), but then go in more detail about categories: first we recall the classical construction of the cluster category associated to acyclic quivers, following Buan, Marsh, Reineke, Reiten, and Todorov, dating back to 2006\footnote{this is a triangulated category, can be thought of as ``without coefficients''.}. Then we consider Jensen, King, and Su's additive categorification of Grassmannian cluster algebras from 2016\footnote{this is a Frobenius category, can be thought of ``with coefficients''.}.

In the last part, Section \ref{Sec:friezes-cat} we explain how to obtain some of the combinatorial frieze types discussed in Section \ref{Sec:friezes} from Grassmannian cluster categories: we define the cluster character and directly prove that it yields a $1-1$-correspondence between Conway--Coxeter friezes and cluster tilting objects in the Grassmannian cluster category $\mc{C}(2,n)$, see Prop.~\ref{Prop:frieze-from-CC}. Combining this with the combinatorial model discussed in Section \ref{Sec:friezes}, we obtain the correspondences summarized in the table in Fig.~\ref{Fig:Correspondence}. \\
In Section \ref{Sec:Pluecker} we survey another application of Grassmannian cluster categories: tame integral $SL_k$-friezes obtained from the so-called Pl\"ucker frieze and categorical interpretations of these. In the following Section \ref{Sub:infinite} we give an overview over infinite Grassmannian cluster structures and draw an analogous table in Fig.~\ref{Fig:Correspondence-infinite}. However, the frieze equivalent in this correspondence table is still missing, and thus at the end we will discuss recent joint work with Esentepe about infinite frieze structures coming from Grassmannian cluster categories of infinite rank. Both sections end with some questions that  might be motivation for further reaching research. \\
Finally, in Section \ref{Sub:other-repthy} there is an outlook to even more constructions of frieze patterns coming from representation theory with  further references.

Before we dive into frieze patterns, a note of warning, that this survey is obviously very biased, reflecting the author's subjective taste! Nonetheless we tried to include as many references and research directions as we are aware of (and time constraints allowed)\footnote{ The intended audience for the talk was mainly people working on (finite dimensional) representations of algebras. However, we hope that the present paper appeals to a wider audience and may serve as an invitation to venture into this area.}. \\
There have been several surveys on friezes more or less recently, all focussing on different aspects of the theory: Morier-Genoud's paper \cite{Crossroads} excellently captures the state of the art from 2015, Pressland's lecture notes give a categorical introduction to Conway--Coxeter-friezes \cite{Pressland-survey}, and Baur \cite{BaurIsfahan} an introduction to Grassmannian cluster categories. Moreover, there is an online resource maintained by Felikson \cite{Felikson-frieze-portal} for material on friezes.

\section{Number friezes} \label{Sec:friezes}

Frieze patterns have not only appeared in art as \emph{friezes} (i.e., broad decorative bands) but also in the study of plane symmetries. There are $7$ \emph{frieze groups}, where a frieze group is defined as a  symmetry of the plane whose subgroup of translations is isomorphic to $\ZZ$, see e.g. \cite[Section 3.7]{Coxeter-Geometry}.

\subsection{Conway--Coxeter friezes} \label{Sub:CCfrieze}

Numeric frieze patterns first seem to appear in Coxeter's paper \cite{Coxeter} in connection with Gauss' pentagramma mirificum. Already in that paper, the main properties of friezes are proven (periodicity in the horizontal direction, the coordinate system, and connections to continued fractions). It is worth to point out that there is a fascinating account about friezes and their connection to three-dimensional polytopes in \cite{Coxeterbook}, where also what is now called \emph{friezes with coefficients} are considered. 

Following Coxeter, the easiest way to introduce friezes, is to write one down, for example, see Fig.~\ref{Fig:glide-coords}.
\begin{figure}[h!]
{\small
 \xymatrix@=0.4em{ 
 &\ldots && 0 && 0 && 0 && 0 && 0 && 0 && 0 && 0 && 0 && 0 && 0 && 0 && 0 & \ldots &\\
   &\ldots & \rot{1} && \rot{1}&& \rot{1} && \rot{1} && \rot{1} && 1 && 1 && 1 && 1 && 1 && 1 && 1 && 1 & \ldots &\\
  &\ldots && \rot{1} && \rot{2} && \rot{2} && \rot{2} &&  1  && 4 &&  1 && 2 && 2 && 2 &&  1  && 4 && 1 & \ldots& \\
  &\ldots & 3 && \rot{1} && \rot{3} &&\rot{3} &&  1  && 3 &&  3 && 1&&  3 && 3 && 1 && 3 && 3    & \ldots& \\
 &\ldots && 2 && \rot{1} && \rot{4} && 1 &&  2  && 2 && 2 && 1 && 4 && 1 &&  2  && 2 && 2 &   \ldots& \\
   &\ldots & 1 && 1 && \rot{1} && 1 && 1 && 1 && 1 &&  1 && 1 && 1 && 1 && 1 && 1  & \ldots & \\
    &\ldots && 0 && 0 && 0 && 0 && 0 && 0 && 0 && 0 && 0 && 0 && 0 && 0 && 0 & \ldots &\\
  }}
  \caption{A Conway--Coxeter frieze, where the fundamental domain (see Fact \ref{Fact:CC-frieze} (2)) is coloured in red.} \label{Fig:glide-coords}
\end{figure}

In Fig.~\ref{Fig:glide-coords} all the entries of this array are positive integers and upon closer inspection one sees that  any four numbers arranged in a diamond 
\begin{align} \label{Eq:diamond}
    \xymatrix@=0.3em{
        &b&\\
        a&&d\\
        &c&
    } 
\end{align}
satisfy 
\[ ad-bc=1 \ . \]

\begin{defi} \label{Def:frieze}
A \textit{finite frieze pattern} (sometimes a \emph{closed frieze}) is a grid of numbers $m_{ij} \neq 0$ in a commutative ring $R$ with $1$  consisting of finitely many rows as follows: the top and the bottom rows are infinite rows of zeroes and the second to top and second to bottom rows are infinite rows of ones. All the other entries $m_{ij}$ are nonzero and satisfy the \textit{frieze rule} or \textit{unimodular rule} \eqref{Eq:diamond}. Thus, a closed frieze looks like this:
\begin{align} \label{Eq:CCfrieze}
    \xymatrix@=0.3em{
    \ldots && 0 && 0 && 0 && 0 &&   \ldots \\
    \ldots & 1 && 1 && 1 && 1 && 1 & \ldots \\
    \ldots && m_{0,2} && m_{1,3} && m_{2,4} && m_{3,5} && \ldots \\
    \ldots & m_{-1,2} && m_{0,3} && m_{1,4} && m_{2,5} && m_{3,6} & \ldots \\
    \ldots && \ddots && \ddots && \ddots && \ddots &&   \ldots \\
    \ldots & m_{-2,w-1} && m_{-1,w}&& m_{0,w+1} && m_{1,w+2} && m_{2,w+3} & \ldots \\
    \ldots && 1 && 1 && 1 && 1 &&   \ldots \\
    \ldots & 0 && 0 && 0 && 0 && 0 & \ldots
    }
\end{align}

The number $w$ of nontrivial rows is called the \emph{width} of the frieze. The first nontrivial row, that is, the row of $m_{i,j}$ with $j-i = 2$ above, is called the \textit{quiddity sequence}. \\ 
If all entries $m_{ij}$ for $i,j \in \ZZ$ are positive integers, then we call the frieze pattern \eqref{Eq:CCfrieze} a \emph{closed integral frieze} or \emph{Conway--Coxeter (=CC) frieze}. Further, if $w= \infty$, that is, there is no lower boundary, then we speak of \emph{infinite friezes}.
\end{defi}

Let us note the following 
\begin{fact} \label{Fact:CC-frieze}
\begin{enumerate}
\item A CC-frieze of width $w < \infty$ is periodic in the horizontal direction with period $n=w+3$, see \cite[Section 6]{Coxeter}. The quiddity sequence $\{ a_i\}_{i=1}^n$ uniquely determines the frieze. Alternatively, a diagonal $m_{i,i+j}$ for $j=2, \ldots, w+1=n-2$ in the frieze determines the remaining entries. One can show that in terms of the quiddity sequence, also already the first $n-2$ entries $a_1, \ldots, a_{w+1}$ determine the frieze uniquely, see e.g. \cite[Lemma 4.1]{FSLotus}.
\item A CC-frieze has a glide symmetry, as can be seen in Fig.~\ref{Fig:glide-coords}: it is invariant under a glide reflection along the horizontal median line of the pattern, see \cite{Coxeter}. Thus friezes consist of a fundamental domain, which can be chosen triangular as the colored entries in the Figure. 
\item  The entries of the frieze can be written as Pl\"ucker coordinates $p_{ij}$ of the Grassmannian $Gr(2,n)$ (see more on this in Example \ref{Ex:Gr(2,n)})\footnote{Coxeter developed the ``coordinate system'' $(i,j)$ for the entries of a frieze in \cite{Coxeter}. In order to save space and to be suggestive for later, we will stay with $m_{ij}$.}. Here note that the same rules as for Pl\"ucker coordinates hold, i.e., we may always assume that $i <j $ and further set $m_{ii}=0$, $m_{i,i+1}=1$.

\end{enumerate}
\end{fact}

Conway--Coxeter friezes were studied by Conway and Coxeter in their remarkable papers \cite{CoCo1, CoCo2}\footnote{These are written in an intriguing way: the first article is a sequence of 34 questions, and the second one consists of the answers of these questions. For a more traditional account of their work see e.g. \cite{Soizic-Henry}.}. They prove a one-to-one correspondence between finite integral friezes and triangulations of finite polygons:
\begin{Thm}[Conway--Coxeter \cite{CoCo1, CoCo2}]
	\label{Thm:CoCo}
There is a bijection between triangulated polygons with $m$ vertices and Conway--Coxeter frieze patterns of width $w=m-3$: 
\begin{enumerate}
\item Let $\mathcal{F}$ be a Conway--Coxeter frieze of width $w$ with the entries $m_{ij}$ as in \eqref{Eq:CCfrieze}. Then the entries $a_i:=m_{i-1,i+1}$ for $ i \in \{ 1, \ldots, m \} $
of $\mathcal{F}$ give a triangulation of an $(w+3)$-gon $P$ with vertices $1, \ldots, w+3$: the integer $a_i$ denotes the number of triangles of the triangulation incident to vertex $i$.
\item Conversely, let $P$ be a triangulated $w+3$-gon. Define the sequence $\{a_i\}_{i=1}^{w+3}$ with $a_i$ as the number of triangles incident  to the vertex $i$. Then the frieze of width $w$ corresponding to $P$ is the one with quiddity sequence $ \{ a_i \}_{i=1}^{w+3} $.
The remaining entries in the frieze can be calculated by the frieze rule and the horizontal periodicity.
\end{enumerate}
\end{Thm}

\begin{Rmk}  \label{Rmk:1-in-frieze}
It was noted by Conway and Coxeter \cite[Question 32]{CoCo1} that the entries $m_{ij}=1$ in the frieze correspond precisely to the diagonals $(i,j)$ in the triangulation.
\end{Rmk}

\begin{Ex} Let us give an example for the Conway--Coxeter bijection: let the number of vertices of the polygon be $m=6$, that is, we consider a triangulated hexagon. Then the theorem says that any triangulation of a hexagon is in $1-1$ correspondence with a CC-frieze of width $3$. There are precisely $14$ triangulations of a hexagon (the number of triangulations of an $m$-gon is given by $C_{m-2}=\frac{1}{m-1} { 2m-4 \choose m-2 }$, the $(m-2)$-nd Catalan number. We write down two of the friezes and their corresponding triangulations:
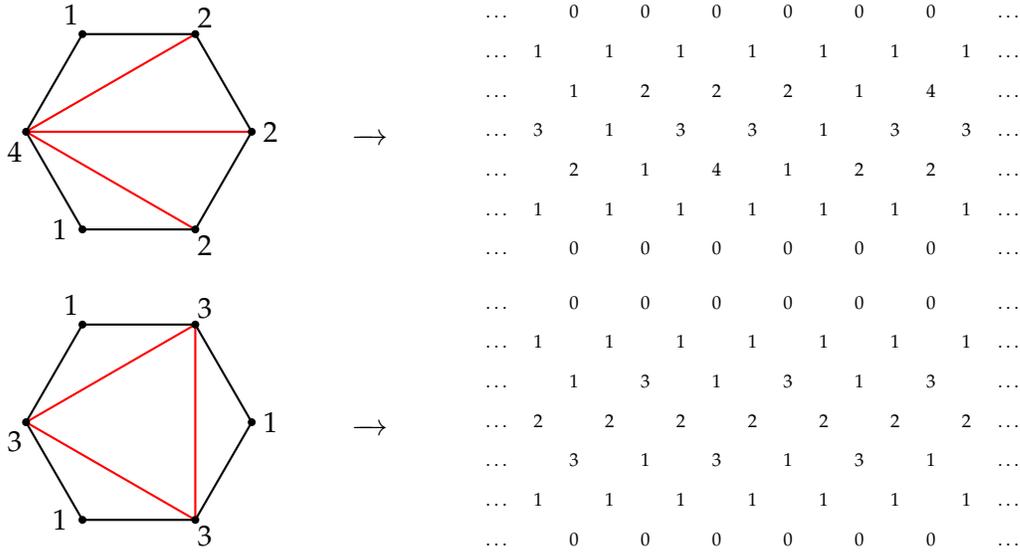
\begin{flushleft}
\begin{minipage}{0.2 \textwidth}\centering
\begin{tikzpicture}[scale=1.5]

\foreach \i in {1,...,6}{
  \coordinate (A\i) at (60*\i:1);
}

\draw[thick] (A1)--(A2)--(A3)--(A4)--(A5)--(A6)--cycle;

\draw[thick,color=red] (A1)--(A3);
\draw[thick, color=red] (A3)--(A5);
\draw[thick, color=red] (A3)--(A6);

        \fill (A2) circle (1pt) node[shift={(300:-0.3)}] {$1$};
\fill (A6) circle (1pt) node[shift={(0:0.25)}] {$2$};
    \fill (A1) circle (1pt) node[shift={(60:0.25)}] {$2$};
     \fill (A5) circle (1pt) node[shift={(120:-0.25)}] {$2$};
   \fill (A4) circle (1pt) node[shift={(180:0.3)}] {$1$};
   \fill (A3) circle (1pt) node[shift={(240:0.3)}] {$4$};

\end{tikzpicture}
\end{minipage}%
\begin{minipage}{0.2 \textwidth}\centering
$\phantom{aaaa}\xrightarrow{ \phantom{aa}}$
\end{minipage}%
\begin{minipage}{0.5 \textwidth} \centering
{\tiny
 \xymatrix@=0.5em{ 
 &\ldots && 0 && 0 && 0 && 0 && 0 && 0 && \ldots &\\
   &\ldots & 1 && 1 && 1 && 1 && 1 && 1 && 1 &  \ldots &\\
  &\ldots && 1 && 2 && 2 && 2 &&  1  && 4 &&  \ldots& \\
  &\ldots & 3 && 1 && 3 && 3 &&  1  && 3 && 3 & \ldots& \\
 &\ldots && 2 && 1 && 4 && 1 &&  2  && 2 &&  \ldots& \\
   &\ldots & 1 && 1 && 1 && 1 && 1 && 1 && 1 &   \ldots & \\
    &\ldots && 0 && 0 && 0 && 0 && 0 && 0 && \ldots &\\
  }
}
\end{minipage}
\end{flushleft}

\begin{flushleft}
\begin{minipage}{0.2 \textwidth}\centering
\begin{tikzpicture}[scale=1.5]

\foreach \i in {1,...,6}{
  \coordinate (A\i) at (60*\i:1);
}

\draw[thick] (A1)--(A2)--(A3)--(A4)--(A5)--(A6)--cycle;

\draw[thick,color=red] (A1)--(A3);
\draw[thick, color=red] (A3)--(A5);
\draw[thick, color=red] (A5)--(A1);

        \fill (A2) circle (1pt) node[shift={(300:-0.3)}] {$1$};
\fill (A6) circle (1pt) node[shift={(0:0.25)}] {$1$};
    \fill (A1) circle (1pt) node[shift={(60:0.25)}] {$3$};
     \fill (A5) circle (1pt) node[shift={(120:-0.25)}] {$3$};
   \fill (A4) circle (1pt) node[shift={(180:0.3)}] {$1$};
   \fill (A3) circle (1pt) node[shift={(240:0.3)}] {$3$};

\end{tikzpicture}
\end{minipage}%
\begin{minipage}{0.2 \textwidth}\centering
$\phantom{aaaa}\xrightarrow{ \phantom{aa}}$
\end{minipage}%
\begin{minipage}{0.5 \textwidth} \centering
{\tiny
 \xymatrix@=0.5em{ 
 &\ldots && 0 && 0 && 0 && 0 && 0 && 0 && \ldots &\\
   &\ldots & 1 && 1 && 1 && 1 && 1 && 1 && 1 &  \ldots &\\
  &\ldots && 1 && 3 && 1 && 3 &&  1  && 3 &&  \ldots& \\
   &\ldots & 2 && 2 && 2 && 2  && 2 && 2 && 2 &  \ldots &\\
 &\ldots && 3 && 1 && 3 && 1 &&  3  && 1 &&  \ldots& \\
   &\ldots & 1 && 1 && 1 && 1 && 1 && 1 && 1 &   \ldots & \\
    &\ldots && 0 && 0 && 0 && 0 && 0 && 0 && \ldots &\\
  }
}
\end{minipage}
\end{flushleft}
\end{Ex}

\subsection{Some generalizations}  \label{Sub:general}
There are various natural generalizations of closed friezes as defined in Def.~\ref{Def:frieze}. We list here a few which will later be considered from a representation theoretic point of view. This list is not exhaustive but gives a first flavour of different frieze constructions.

\subsubsection*{Infinite periodic friezes} The first generalization has already appeared in \cite{CoCo1} in Question 16: an integral frieze of infinite width. In this example, the quiddity sequence is still periodic. There have been some studies about periodic integral friezes, see \cite{KarinSurvey} for a recent account: Tschabold has shown that any periodic infinite frieze of period $n$ comes from a triangulation of a punctured disc with $n$ marked points on the boundary \cite[Theorem 3.6]{Tschabold} or alternatively, it comes from a triangulation of the annulus \cite[Theorem 4.6]{BPT}. Let us note that for infinite friezes one can also define a \emph{growth coefficient} of the entries, and there are some results on the growth of infinite friezes, e.g. in \cite{BaurFellnerParsonsTschabold}. \\ 
 In Section \ref{Sub:infinite} we consider infinite friezes with non-periodic quiddity sequences.

\subsubsection*{Friezes with coefficients} This generalization has also already been considered by Coxeter \cite[Section 5.1]{Coxeterbook} but also was featured prominently in Propp's rich article \cite{Propp}. The idea is the following: from the combinatorial model for a CC-frieze, i.e., the entries of the frieze correspond to matching numbers in a triangulated polygon, we have that all the boundary edges of said polygon have value (or weight) $1$. In the coordinate system \eqref{Eq:CCfrieze} this means that $m_{i,i+1}=1$ for any $i$. But one can also vary the weights of the boundary edges (cf.~\cite[Def.~2.1]{CuntzHolmJorgensen}): 
\begin{defi} \label{Def:frieze-coeffs}
Let $R$ be an integral domain 
and let $w \in \ZZ_{\geq 0}$. A \emph{frieze with coefficients} $\{c_{i,i+1}\}_{i \in \ZZ}$ of width $w$ is an array that is infinite in both horizontal directions
\begin{align} 
    \xymatrix@=0.3em{
    \ldots && m_{1,1} && m_{2,2} && m_{3,3} && m_{4,4} &&   \ldots \\
    \ldots & m_{0,1} && m_{1,2} && m_{2,3} && m_{3,4} && m_{4,5} & \ldots \\
    \ldots && m_{0,2} && m_{1,3} && m_{2,4} && m_{3,5} && \ldots \\
    \ldots & m_{-1,2} && m_{0,3} && m_{1,4} && m_{2,5} && m_{3,6} & \ldots \\
    \ldots && \ddots && \ddots && \ddots && \ddots &&   \ldots \\
    \ldots & m_{-2,w-1} && m_{-1,w}&& m_{0,w+1} && m_{1,w+2} && m_{2,w+3} & \ldots \\
    \ldots && m_{-2,w} && m_{-1,w+1} && m_{0,w+2} && m_{1,w+3} &&   \ldots \\
    \ldots & m_{-3,w} && m_{-2,w+1} && m_{-1,w+2} && m_{0,w+3} && m_{1,w+4} & \ldots
    }
\end{align}
such that all the conditions
\begin{align} \label{Eq:coeff1} m_{i,i}=m_{i,i+w+3}=0 \ ,  && m_{i,i+1}=c_{i,i+1}\neq 0  \in R \ , \\
\label{Eq:coeff2}  m_{i,j}  \in R \text{ for } i<j<i+w+4\  , &&  \\
\label{Eq:frieze-coeff}   m_{i,j}m_{i+1,j+1}-m_{i+1,j}m_{i,j+1}=m_{i,i+1}m_{j+1,w+j+4}\  , &&
\end{align}
are satisfied. \end{defi}
Note that conditions \eqref{Eq:coeff1}, \eqref{Eq:coeff2}, and \eqref{Eq:frieze-coeff} imply (similar as in \cite[Remark 2.2 (2)]{CuntzHolmJorgensen}) that
$$c_{i,i+1}=m_{i,i+1}=m_{i+1,w+i+4} \ , $$
that is, the frieze is periodic in the horizontal direction and the frieze rule \eqref{Eq:frieze-coeff} can also be written as
$$m_{i,j}m_{i+1,j+1}-m_{i+1,j}m_{i,j+1}=c_{i,i+1}c_{j,j+1} \ . $$

Note that in the paper by Cuntz, Holm, and J{\o}rgensen these friezes are studied from a combinatorial point of view. Similar ideas have also been considered by Fontaine for friezes with integral coefficients \cite{Fontaine}.

\subsubsection*{$SL_k$-friezes and $SL_k$-tilings} These have first appeared $1972$ in work of Cordes and Roselle \cite{CordesRoselle} (with extra conditions on the minors of the first $k \times k$-diamonds) and seem to not have been further researched until the introduction of cluster algebras. In particular, $SL_3$-friezes were discussed (in the guise of $2$-friezes) in \cite[Section 6]{Propp}, in \cite{BergeronReutenauer} $SL_k$-tilings were considered, and  $SL_k$-friezes are studied more systematically in \cite{MGOT12, mg12, mgost14, Cuntz}. Here for the first time we need an extra condition on $(k+1) \times (k+1)$ determinants in the patterns to make sure that the patterns have a controllable behaviour.

\begin{defi}
An \emph{$SL_k$-frieze} over an integral domain $R$ is an array of numbers with finitely many rows:  
\begin{equation} \label{Eq:SLkfrieze}
\xymatrix@C=.001em@R=.05em{
\dots & &  0 && 0 && 0 && 0 && 0 && \dots  \\ 
&\ddots & & \ddots  &   & \ddots  & & \ddots && \ddots && \ddots && \\ 
\dots&  0 && 0 && 0 && 0 && 0 && 0 &\dots  \\
\dots & & 1 && 1 && 1 && 1 && 1 &&  \dots \\
\dots &   && m_{-1,-1} &&m_{00} && m_{11} &&  m_{22} && &\dots \\
 & &   & \ddots& \ddots && \ddots && \ddots \\
  \dots  && m_{-2,w-4}  && m_{-1,w-3} && m_{0,w-2} && m_{1,w-1} && m_{2,w} && \dots \\
   &   && m_{-2,w-3} && m_{-1,w-2} && m_{0,w-1} && m_{1,w}  &&  \dots\\
   \dots && 1 && 1 && 1 && 1 && 1 && \cdots && \\ 
\dots &   0 && 0 && 0 && 0 && 0 && 0 & \dots&   \\
&&  \ddots &   & \ddots  && \ddots && \ddots &&   \\
\dots && 0 && 0 & &  0 &&  0 && 0 && \dots &&  &&  \\
}
\end{equation}
The first $k-1$ rows consist of $0$'s followed by a row of $1$'s from top and inversely at the bottom,
and of $w \geq 1$ rows of elements $m_{ij}\in R$ in between, such that every $k\times k$-diamond of entries of the frieze has determinant $1$. The integer $w$ is again called the \emph{width} of the frieze.\footnote{Similar to \cite[Section 3]{Crossroads}, we chose a slightly different labeling convention than in Definition \ref{Def:frieze}: the entries form a grid as before, but here they are not directly indexed by Pl\"ucker coordinates of some $Gr(2,n)$. } \\
The $SL_k$-frieze is called \emph{tame} if all $(k+1)\times (k+1)$-diamonds have determinant $0$. In general, we define an $l \times l$-diamond of the frieze, where $1 \leq l \leq k+1$, as the matrix having $l$ successive entries of the frieze on its diagonals and all other entries above and below the diagonal accordingly. Finally, if all non-trivial entries $m_{ij}$ with $0 \leq i,j \leq w$ are positive integers, we call the pattern \eqref{Eq:SLkfrieze} an \emph{integral} 
$SL_k$-frieze. 
\end{defi}
We will see in Section \ref{Sec:Pluecker} how these friezes can be interpreted as evaluations of so-called Pl\"ucker friezes, whose entries are Pl\"ucker coordinates in the Grassmannian $Gr(k, n)$, where $n=w+k+1$.

\subsubsection*{Other types of friezes} \label{Subsub:other}

There are also generalizations in a different direction, where the determinant condition is changed. For example, additive friezes have been considered by Coxeter and Rigby, also by Shephard \cite{coxeterrigby, Shephard} in the $1970$s - there the diamond rule \eqref{Eq:diamond} is altered to 
$$a+d=c+d+1 \ . $$
For these friezes a representation theoretic  description would be  interesting. There is also the variant of tropical friezes, which have been studied from a representation theoretic point of view see \cite{RingelTropical, GuoIMRN}. For even more elaborate rules, let us mention symplectic friezes \cite{MG19}, Heronian friezes \cite{FominSetiabrata}, and most recently noncommutative friezes \cite{CHJ25}.

\section{Cluster algebras and cluster categories} \label{Sec:cluster-algs-cats}

In this section we will explain how to get from cluster algebras and categories to friezes. The basic observation for this is again due to Coxeter: every entry in a closed frieze is a continuant polynomial in elements of the quiddity sequence, see Prop.~\ref{Prop:quiddityFrieze}. Continuants appear as exchange polynomials in cluster algebras of type $A$, which will also be defined below. Further, we will look at one prominent example of algebras with a cluster structure: the homogeneous coordinate ring of the Grassmannian $Gr(k,n)$. Then we explain the relation to cluster categories, in particular to Grassmannian cluster categories, which were introduced by Jensen, King, and Su \cite{JKS16}. 

\begin{defi} A \emph{continuant} of order $n$ is the determinant of a tri-diagonal matrix of the form 
$$\begin{pmatrix} y_1 & b_1 & 0 & \cdots & \cdots & \cdots \\
c_1 & y_2 & b_2 & 0 & \cdots & \cdots  \\
0 & c_2 & y_3 & b_3 & 0 & \cdots \\
\vdots & \ddots &  \ddots & \ddots &  \ddots & \ddots \\
0 & \cdots & 0  &  c_{n-2} & y_{n-1} & b_{n-1} \\
0 & \cdots & \cdots  & 0 & c_{n-1} & y_n \end{pmatrix} \ . $$
If $b_i=c_i=1$ for all 
$ i \in \{ 1, \ldots , n-1 \} $, then denote this continuant by $P_n(y_1, \ldots, y_n)$. We set $P_0:=1$. 
\end{defi}

\begin{Prop}[Section (6.6) of \cite{Coxeter}] \label{Prop:quiddityFrieze}
Let $\mc{F}$ be a closed frieze of the form \eqref{Eq:CCfrieze} and denote the quiddity sequence by $a_i:=m_{i-1,i+1}$ for $i \in \ZZ$. Then 
$$m_{i,j}=P_{j-i-1}(a_{i+1}, \ldots, a_{j-1}) \ .$$
This means that any such $\mc{F}$ starts as follows
{\small
\[
 \xymatrix@=0.3em{ 
 &\ldots && 0 && 0 && 0 && 0 &&   \ldots &\\
   &\ldots & 1 && 1 && 1 && 1 && 1    &  \ldots &\\
  &\ldots && a_1 && a_2 && a_3 && a_4 &&      \ldots& \\
  &\ldots & \ldots  && P_2(a_1,a_2) && P_2(a_2,a_3) && P_2(a_3,a_4) &&        & \ldots& \\
 &\ldots && \ldots && P_3(a_1,a_2,a_3) && P_3(a_2,a_3,a_4) && \ldots  &&     \ldots& \\
   &\ldots & \ddots && \ddots && \ddots  &&  \ddots &&  \ddots&   \ldots & 
}
  \]
  }
\end{Prop}
Note here that because of the horizontal periodicity of the frieze, see Fact \ref{Fact:CC-frieze} (1), we have that $a_i=a_{i+w+3}$ and thus one only has to compute the continuants for $a_i$, $1 \leq i \leq w+3$ in order to obtain all the entries $m_{ij}$ in the frieze.
Continuant polynomials are a classical topic in the study of determinants and have also appeared in connection with continued fractions \cite{Muir,FareyBoat}.

\subsection{Cluster algebras}  \label{Sec:ClusterAlgs}

Cluster algebras were introduced by Fomin and Zelevinsky in the early 2000s, see \cite{FZ2002}, for more details and the general theory see e.g~\cite{FWZ2016,FWZ2017,FWZ2020,NakanishiBook}. 

In this paper we will mostly be interested in Grassmannian cluster algebras and categories, that means in particular that we will define cluster algebras of geometric type with coefficients, mostly following \cite{FZ2007}.

Let $ (\PP, \oplus, \cdot ) $ be a semifield, i.e., it fulfills the same axioms as a field with the possible exception for the existence of neutral element with respect to $ \oplus $ and for the existence of an inverse element with respect to $ \oplus $. 
 Fix a positive integer $N \in \ZZ_+ $ and a field $K$\footnote{In order to avoid possible pathologies, we will assume that $K$ is algebraically closed throughout. Mostly we will take $K=\CC$.}. 
The {\em ambient field} for a cluster algebra $ \cA $ of rank $ N $ is a field $ \cF$ isomorphic to the field of rational functions in $ N $ independent variables with coefficients in $ \KK \PP $. 
\\

Further recall the notion of a tropical semifield \cite[Definition~2.2]{FZ2007}: Let $ ( \PP, \cdot ) $ be a multiplicative free abelian group with a finite free set of generators 
	$ (p_1, \ldots, p_\ell) $.
	We endow $ \PP $ additionally with the auxiliary addition $ \oplus $ given by
	\[
	\prod_{i=1}^\ell p_i^{a_i} \oplus \prod_{i=1}^\ell p_i^{b_i} := \prod_{i=1}^\ell p_i^{\min\{ a_i, b_i\}}.
	\]
	Using this addition, $ (\PP, \oplus, \cdot ) $ is a semifield,
	which is usually denoted by $ \operatorname{Trop}(p_1, \ldots , p_\ell ) $, the {\em tropical semifield}. 

An $ N $-tuple $ \x = (x_1, \ldots, x_N) $ of elements in $ \cF $ is called a {\em free generating set} if 
$ x_1, \ldots, x_N $ are algebraically independent over $ \KK\PP $ and $ \cF= \KK\PP(x_1, \ldots, x_N) $. 
Moreover, recall that an $ N \times N$ integer matrix $ B = ( b_{ij} ) $ is called  {\em skew-symmetrizable}
if there exist positive integers $ d_1, \ldots, d_N $ such that, for all $ i, j \in \{ 1, \ldots, N \} $, we have
\[
	d_i b_{ij} = - d_j b_{ji} \ . 
\] 
In particular, a skew-symmetrizable matrix is sign-skew-symmetric. 

We will define cluster algebras via extended exchange matrices. Therefore let
	\begin{itemize}
		\item 
		$ \y = (y_1, \ldots, y_N) $ be an $N $-tuple of elements in $ \PP $, called the {\em coefficient tuple},
		
		\item 
		$ B = (b_{ij}) $ be a skew-symmetrizable $ N \times N $ integer matrix,
		called the {\em exchange matrix},
		
		\item 
		$ \x = (x_1, \ldots, x_N) $ be an $ N $-tuple of elements in $ \cF $ forming a free generating set, called the {\em cluster}. 
	\end{itemize}
	
	We are renaming the generators $ \p = (p_1, \ldots, p_\ell) $ of the tropical semifield of $ \PP $ by calling them $ x_{N+1}, \ldots, x_{M} $, where $ M := N + \ell $,
	and set 
	\[
	\widetilde{\x} := (x_1, \ldots, x_N, \ldots, x_M). 
	\]
	Since the coefficients $ y_1, \ldots, y_N $ are Laurent monomials in $ x_{N+1}, \ldots, x_M $, we may define the integers 
	\[
	b_{ij} \in \ZZ,
	\ \ \
	\mbox{for } j \in \{ 1, \ldots, N \} 
	\mbox{ and } i \in \{ N+1, \ldots, M \}
	\]
	via 
	\begin{equation}
		\label{eq:y_j}
		y_j = \prod_{i=N+1}^M  x_i^{b_{ij}} \ . 
	\end{equation}
	
	In other words,
	we are extending the $ N \times N $ exchange matrix $ B $ to an {\em extended exchange matrix $ \widetilde B $} of size $ M \times N $,
	where 
	\[
	\widetilde B := ( b_{ij})_{ i \in \{ 1, \ldots, M \}, j \in \{1, \ldots, N\}} \ .
	\]
	Thus, the new entries below $ B $ (i.e., those $ b_{ij} $ with $ i \geq N + 1 $) are determined by the exponents of $ x_i $  in \eqref{eq:y_j}. 
	The pair $ \Sigma = (\widetilde{\x},\widetilde{B}) $ is called a {\em labeled seed}.

	\begin{defi} \label{Def:seed_mutation}
		Let $ \Sigma = (\widetilde \x, \widetilde B) $ be a labeled seed and $ k \in \{ 1, \ldots, N \} $.  
		The {\em seed mutation} $ \mu_k $ in direction $ k $ transforms $ (\widetilde \x, \widetilde B) $ into the labeled seed $ \mu_k (\widetilde \x, \widetilde B)  := (\widetilde \x', \widetilde B') $,
		where:
		\begin{enumerate}
			\item 
			a new cluster $ \x' = (x_1', \ldots, x_N') $ is introduced via $ x_j' := x_j $ if $ j \neq k $ 
			and $ x_k' \in \cF $ is defined by the {\em exchange relation}:
			\begin{equation}
				\label{eq:exchange_TildeB_1}
				x_k x_k'  
				= 
				\prod\limits_{\substack{i=1 \\ b_{ik} > 0 }}^M x_i^{b_{ik}} + \prod\limits_{\substack{i=1 \\ b_{ik} < 0 }}^M x_i^{-b_{ik}} 				\ .
			\end{equation}
			Notice that the appearing exponents are the entries of the $ k $-th column of $ \widetilde{B} $. The polynomial $x_k x_k'  
				- 
				\prod\limits_{\substack{i=1 \\ b_{ik} > 0 }}^M x_i^{b_{ik}} - \prod\limits_{\substack{i=1 \\ b_{ik} < 0 }}^M x_i^{-b_{ik}}$ is sometimes called \emph{exchange polynomial}.
			
			\item 
			the new exchange matrix 
			$ \widetilde B' = (b_{ij}') $ is of the same size as $ \widetilde B $ and its entries are determined by
			\[
			b_{ij}' := 
			\begin{cases}
				- b_{ij} & \mbox{if } i = k \mbox{ or } j = k \ ,
				\\
				b_{ij} + \operatorname{sgn}(b_{ik}) [ b_{ik} b_{kj} ]_+
				& \mbox{otherwise} \ .
			\end{cases}
			\]	
			Here $[a]_+:=\max(a,0)$ for any real integer $a$.
		\end{enumerate}
		Since we are only allowed to mutate in directions $ \{ 1, \ldots, N \} $,
		the variables $ (x_1, \ldots, x_N) $ are sometimes called the {\em mutable variables} of $ (\widetilde{\x}, \widetilde{B}) $,
		while $ (x_{N+1}, \ldots, x_M) $ are called the {\em frozen variables} or {\em coefficient variables} of $ (\widetilde{\x}, \widetilde{B}) $, cf.~e.g.~\cite[Definition~3.1.1]{FWZ2016}.
	\end{defi}

Note that seed mutation is well-defined, i.e., $ \mu_k (\Sigma) $ is again a labeled seed in $ \cF $.
Furthermore, $ \mu_k $ is an involution, that is, we have $ \mu_k ( \mu_k (\Sigma) ) =  \Sigma $.

In order to finally define cluster algebras, we further need \emph{cluster patterns}: Therefore, recall that the {\em $N$-regular tree} $ \TT_N $ is the infinite simply laced tree such that every vertex $ v $ has $ N $ edges and the latter are labeled with the numbers $ 1, \ldots, N $. 
Then a {\em cluster pattern} is an assignment of a labeled seed $ \Sigma_t = ( \widetilde \x_t, \widetilde B_t) $ to every vertex $ t \in \TT_N $ such that the following property holds:
	\\
	For every pair of vertices $ t, t' \in \TT_N$ which are joined by an edge, say with label $ k $, we have 
	$ \mu_k (\Sigma_t) = \Sigma_{t'} $. 
Note that we also have $ \mu_k (\Sigma_{t'}) = \Sigma_{t} $ in the above situation since $ \mu_k $ is an involution.

\begin{defi}[{\cite[Definition~2.11]{FZ2007}}] 
	\label{Def:ClusterAlgebra}
	For a cluster pattern determined by the seeds $ \Sigma_t = ( \widetilde \x_t, \widetilde B_t)$, $ t \in \TT_N $, 
	we set
	\[
		\cX := \bigcup_{t \in \TT_N } \x_t = \{ x_{i,t} \mid t \in \TT_N, i \in \{ 1 , \ldots, N \} \},
	\] 
	where $ \x_t = (x_{1,t}, \ldots, x_{N,t} ) $.
	We call $ x_{i,t} \in \cX $ the {\em cluster variables}.
	\\
	The {\em cluster algebra} $ \cA $ associated to the given cluster pattern is the $ \KK\PP $-subalgebra of the ambient field $ \cF $ generated by all cluster variables,
	\[
		\cA = \KK\PP [\cX] \ . 	
	\]
	For a labeled seed $ \Sigma =  (\widetilde \x , \widetilde B) = ( \widetilde \x_t, \widetilde B_t) $ with $ t \in \TT_n $ a fixed value,
	we also use the notation
	\[
		\cA = \cA (\Sigma) = \cA(\widetilde \x, \widetilde B) \ . 
	\]
\end{defi}

\medskip 

By the last part, we may define a cluster algebra also by giving a labeled seed $ \Sigma $ in $ \cF $ and to create then a cluster pattern by repeatedly mutating in all possible directions.

\begin{Rmk}[Quiver perspective in the case of skew-symmetric exchange matrices]
	\label{Rk:Quiver_view}
	Let $ \Sigma $ be a labeled seed. 
	If the exchange matrix $ B $ is skew-symmetric, 
	there is an interpretation of $ \Sigma $ in terms of quivers\footnote{This will always be the case for the cluster algebras occurring in the present article.}. 
	\\
	The quiver $ \cQ(\widetilde{B}) $ associated to $ \widetilde{B} = (b_{ij} ) $ is the finite directed graph 
	with $M$ vertices, labeled by $ 1, \ldots, M $
	and such that,
	for $ i \geq j $ there are $ |b_{ij}| \in \ZZ_{\geq 0} $ many pairwise different arrows going from the vertex $ i $ to the vertex $ j $, if $ b_{ij} \geq 0 $, 
	respectively, going from $ j $ to $ i $, if $ b_{ij} < 0 $.
	In order to distinguish mutable and frozen variables in $\cQ(\widetilde{B})$, the first are marked with circles, while the latter are marked with squares.
	\\
For example, if the extended exchange matrix is
	\[
	\widetilde{B} = 
	\begin{pmatrix}
	0 & 1  \\
	-1 & 0 \\
	-1 & 1 \\
	\end{pmatrix}
	\] 
	the following quiver $ \cQ(\widetilde{B}) $: 
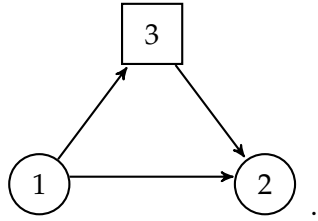
	\begin{center} 
		\begin{tikzpicture}[->,>=stealth',shorten >=1pt,auto,node distance=1.75cm, thick,main node/.style={circle,draw, minimum size=8mm},frozen node/.style={rectangle,draw=black, minimum size=8mm}
		]

		\node[main node] (1) at (0,0) {$ 1 $};
		\node[main node] (2) at (3,0) {$ 2 $};
		\node[frozen node] (3) at (1.5,2) {$ 3 $};
		
		\draw[->] (0.4,0.1) -- (2.62,0.1);
		
		\path 
		(1) edge (3);

		\path 
		(3) edge (2);

		\end{tikzpicture}
		\ .
	\end{center}
	
	The exchange relation \eqref{eq:exchange_TildeB_1} then translates to	
	\[ 
	x_k x_k' = 	
	\prod\limits_{i \to k } x_i^{b_{ik}} + \prod\limits_{k \to i} x_i^{-b_{ik}} 
	\ .
	\]
	Furthermore, there is also a mutation rule for the quiver $ \cQ(\widetilde{B}) $, which is compatible with the mutation of the matrix described in Definition~\ref{Def:seed_mutation}(2). We will not further discuss the notion of quiver mutation in this survey and refer to \cite[Chapter 2]{FWZ2016} for more details.
\end{Rmk}


\begin{Ex} 
Consider the simplest non-trivial case without  coefficients, that is the seed $ (\x, B) $,
where $ \x = (x_1, x_2 ) $, $B=\begin{pmatrix} 0 & 1 \\ -1 & 0 \end{pmatrix}$, and the corresponding $ Q = Q(B) $ is an $A_2$-quiver,  
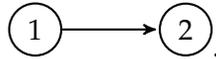
\begin{center} 
	\begin{tikzpicture}[->,>=stealth',shorten >=1pt,auto,node distance=2cm, thick,main node/.style={circle,draw}]
	
	\node[main node] (1) {$1$};
	\node[main node] (2) [right of=1] {$2$};
	
	\path
	(1) edge (2);
	\end{tikzpicture}.
\end{center}
Note here that $N=2$ and $M=0$.
The following table (cf.~\cite[Example 2.7]{BFMS}) describes the behaviour of $ (\x, Q ) $ along repeated mutation:
\[
\begin{array}{|r|c|c|}
\hline 
\mbox{quiver} \hoehe \hspace{1.5cm} & \mbox{cluster} & \mbox{expression in terms of initial cluster $ \x = (x_1, x_2) $}
\\
\hhline {|=|=|=|}
Q = \atwo 
\hoehe 
& 
(x_1, x_2) 
& -
\\
\hline 
\mu_1 (Q) = \aTwo 
\hoehe 
& 
(x_1^{(1)}, x_2^{(1)})
&
x_1^{(1)} = \dfrac{1+ x_2}{x_1},
\hspace{15pt}
x_2^{(1)} = x_2
\\
\hline 
\mu_2 (Q) = \aTwo 
\hoehe 
& 
(x_1^{(2)}, x_2^{(2)})
&
x_1^{(2)} = x_1, 
\hspace{15pt} 
x_2^{(2)} = \dfrac{x_1 + 1}{x_2}
\\
\hline 
\mu_1(\mu_1 (Q)) = \atwo 
\hoehe 
& 
(x_1^{(11)}, x_2^{(11)})
&
x_1^{(11)} = x_1, 
\hspace{15pt} 
x_2^{(11)} = x_2
\\
\hline 
\mu_2(\mu_2 (Q)) = \atwo 
\hoehe 
& 
(x_1^{(22)}, x_2^{(22)})
&
x_1^{(22)} = x_1, 
\hspace{15pt} 
x_2^{(22)} = x_2
\\
\hline 
\mu_1(\mu_2 (Q)) = \atwo 
\hoehe 
& 
(x_1^{(12)}, x_2^{(12)})
&
x_1^{(12)} = \dfrac{1+x_1+x_2}{x_1 x_2}, 
\hspace{15pt} 
x_2^{(12)} = \dfrac{x_1 + 1}{x_2}
\\
\hline 
\mu_2(\mu_1 (Q)) = \atwo 
\hoehe 
& 
(x_1^{(21)}, x_2^{(21)})
&
x_1^{(21)} = \dfrac{1 + x_2}{x_1}, 
\hspace{15pt} 
x_2^{(21)} = \dfrac{1+x_1+x_2}{x_1 x_2}
\\ 
\hline 
\end{array} 
\] 
Here we have
\[
\cA:= 
\cA (\x, Q) 
= K \left[x_1, x_2, \frac{1+x_1}{x_2}, \frac{1+x_2}{x_1}, \frac{1+x_1+ x_2}{x_1 x_2} \right]
=
K \left[x_1, \frac{1+x_1}{x_2}, \frac{1+x_2}{x_1} \right]
\cong 
\]\[	\cong
K[u,v,w]/ \langle uvw - u - v - 1 \rangle
 \ .  
\] 
This is a hypersurface singularity, whose defining equation is the continuant $P_{3}(u,v,w)-1$.

\end{Ex}

\begin{Ex} Let more generally $Q$ be the quiver of type $A_{n-3}$ with linear orientation:
	
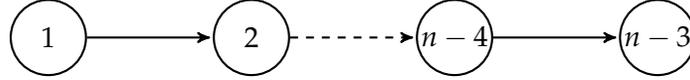
\begin{center} 
	\begin{tikzpicture}[->,>=stealth',shorten >=1pt,auto,node distance=2.7cm, thick,
  main node/.style={circle,draw,minimum size=10mm,inner sep=0pt}]
		
  \node[main node] (1) {$1$};
  \node[main node] (2) [right of=1] {$2$};
  \node[main node] (n-4) [right of=2] {$n-4$};
  \node[main node] (n-3) [right of=n-4] {$n-3$};
		
  \path
    (1) edge (2);
  \path[dashed]
    (2) edge (n-4);
  \path 
    (n-4) edge (n-3);
\end{tikzpicture}
	\end{center}
 Since any quiver of type $A_{n-3}$ is mutation equivalent to $Q$, any cluster algebra of type $A_{n-3}$ is isomorphic to $\mc{A}(\x, Q)$ where $\x=(x_1, \ldots, x_{n-3})$ and we do not have frozen variables. Since $Q$ is acyclic, the cluster algebra $\mc{A}(\x,Q)$ is generated by the initial seed $\x$ together with $n-3$ new cluster variables obtained by mutating $Q$ in each direction once:
 \begin{align*}
 \mc{A}(\x,Q) & =K[x_1, \ldots, x_{n-3},z_1, \ldots, z_{n-3}]/\langle x_1z_1-x_2-1, x_{n-3}z_{n-3}-x_{n-4}-1, \\
  & \phantom{=K[x_1, \ldots, x_{n-3},z_1, \ldots, z_{n-3}]/\langle l}x_iz_i-x_{i-1} - x_{i+1}, 2 \leq i \leq n-4\rangle \\
 & \cong K[u_1,\ldots, u_{n-2}]/(P_{n-2}(u_1, \ldots, u_{n-2})-1) \ .
 \end{align*}
Note that this is again a hypersurface ring\footnote{In \cite{BFMS} the singularities of all cluster algebras of finite cluster type and with trivial coefficients were classified.}.
\end{Ex}

\begin{Ex} \label{Ex:Gr(2,n)}
The  homogeneous coordinate ring of the Grassmannian $Gr(2,n)$ of $2$-subspaces in $\CC^n$ can be described with Pl\"ucker coordinates, as
$$\mc{A}(2,n):=\CC[Gr(2,n)]=\CC[p_{ij}: 1 \leq i < j \leq n]/\langle \text{Pl\"ucker relations}\rangle \ ,$$
where the Pl\"ucker relations are given as
\begin{equation} \label{Eq:pluecker-rel}
p_{ab}p_{cd} -p_{ac}p_{bd} - p_{ad} p_{bc} 
\end{equation}
for any $a < c < b <d $ in the cyclic ordering (that is, consider the vertices cyclically ordered, say clockwise, in an $n$-gon $P_n$ as in the graphics below): 
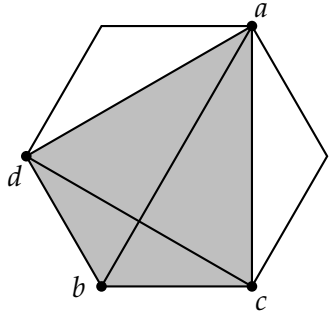
\begin{center}
\begin{tikzpicture}[scale=2]

\foreach \i in {1,...,6}{
  \coordinate (A\i) at (60*\i:1);
}

\fill[lightgray] (A1)--(A3)--(A4)--(A5)--cycle;
\draw[thick] (A1)--(A2)--(A3)--(A4)--(A5)--(A6)--cycle;

\draw[thick] (A1)--(A3);
\draw[thick] (A1)--(A4);
\draw[thick] (A1)--(A5);
\draw[thick] (A3)--(A5);

    \fill (A1) circle (1pt) node[shift={(60:0.25)}] {$a$};
     \fill (A5) circle (1pt) node[shift={(120:-0.25)}] {$c$};
   \fill (A4) circle (1pt) node[shift={(180:0.3)}] {$b$};
   \fill (A3) circle (1pt) node[shift={(240:0.3)}] {$d$};

\end{tikzpicture}
\end{center}
In this example the Pl\"ucker coordinates can be seen as arcs in an $n$-gon, where the $p_{i,i+1}$ are the boundary arcs. The ring $\cA(2,n)$ has the structure of a cluster algebra with the clusters corresponding to triangulations (=maximal sets of non-crossing diagonals) of $P_n$.  The boundary edges are the frozen variables and the diagonals are the mutable variables. Mutation in the polygon is given by flipping an arc $(a,b)$ to an arc $(c,d)$ as in Fig.~\ref{Fig:flip-diag}, which corresponds to the Pl\"ucker relation \eqref{Eq:pluecker-rel}. More precisely: if $(a,b)$ is an arc in a fixed triangulation $T$ of $P_n$, then it is a diagonal in a unique quadrilateral (gray in the picture) with vertices $a,c,b,d$ in clockwise order. Mutation means to switch the diagonal $(a,b)$ with $(c,d)$ and to obtain the new triangulation $\mu_{(a,b)}(T)=(T \backslash \{(a,b)\} ) \cup \{ (c,d) \}$.
\begin{figure}[h!] 
\begin{minipage}{0.4 \textwidth}\centering
\begin{tikzpicture}[scale=2]

\foreach \i in {1,...,6}{
  \coordinate (A\i) at (60*\i:1);
}

\fill[lightgray] (A1)--(A3)--(A4)--(A5)--cycle;
\draw[thick] (A1)--(A2)--(A3)--(A4)--(A5)--(A6)--cycle;

\draw[thick] (A1)--(A3);
\draw[thick,color=red] (A1)--(A4);
\draw[thick] (A1)--(A5);

    \fill (A1) circle (1pt) node[shift={(60:0.25)}] {$a$};
     \fill (A5) circle (1pt) node[shift={(120:-0.25)}] {$c$};
   \fill (A4) circle (1pt) node[shift={(180:0.3)}] {$b$};
   \fill (A3) circle (1pt) node[shift={(240:0.3)}] {$d$};

\end{tikzpicture}
\end{minipage}
\begin{minipage}{0.4 \textwidth} \centering
\begin{tikzpicture}[scale=2]

\foreach \i in {1,...,6}{
  \coordinate (A\i) at (60*\i:1);
}

\fill[lightgray] (A1)--(A3)--(A4)--(A5)--cycle;
\draw[thick] (A1)--(A2)--(A3)--(A4)--(A5)--(A6)--cycle;

\draw[thick] (A1)--(A3);
\draw[thick] (A1)--(A5);
\draw[thick,color=red] (A3)--(A5);

    \fill (A1) circle (1pt) node[shift={(60:0.25)}] {$a$};
     \fill (A5) circle (1pt) node[shift={(120:-0.25)}] {$c$};
   \fill (A4) circle (1pt) node[shift={(180:0.3)}] {$b$};
   \fill (A3) circle (1pt) node[shift={(240:0.3)}] {$d$};

\end{tikzpicture}
\end{minipage}
\caption{Mutation in polygon: flip of diagonal.} \label{Fig:flip-diag}
\end{figure}
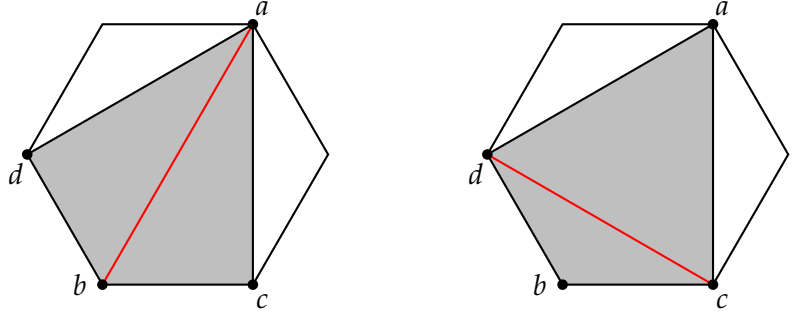 
Note here that we will often speak of Pl\"ucker coordinates geometrically: $p_{ab}$ and $p_{cd}$ are \emph{crossing} if either 
\[a < c < b < d \text{ or } c < a < d < b \]
in the cyclic ordering. If  $p_{ab}$ and $p_{cd}$ are not crossing, then we call them \emph{non-crossing} (or \emph{compatible}).

This example is described in detail in \cite[Section 12.2]{FZ2002}, as well as in \cite[Section 1.2]{BaurIsfahan}. In \cite{Scott06} it is generalized to $\CC[Gr(k,n)]$, also see Example \ref{Ex:Grassmannian-kn} below. 
\end{Ex}

Let  us note here that any cluster in $\cA (2,n)$ corresponds to a triangulation of $P_n$ and conversely, any triangulation of $P_n$ yields a cluster in $\CC[Gr(2,n)]$. One can also associate a quiver to a triangulation, this is explained (without coefficients) in Section \ref{Sub:triangulation-quiver}. For the quivers including frozen variables, see \cite{Scott06}.
Setting $p_{ii}=0$ and mapping $p_{ij}$ to $m_{ij}$ we obtain a frieze with coefficients with entries Pl\"ucker variables, see Def.~\ref{Def:frieze-coeffs}, this way: the frieze relations are satisfied, as they are just the Pl\"ucker relations for quadrilaterals with vertices $i, i+1, j, j+1$. The obtained frieze is also the specialized Pl\"ucker frieze $s\mc{P}_{2,n}$, cf.~Def.~\ref{Def:Pluecker-frieze}. \\
Thus any positive integer point on the Grassmannian will yield a closed integral frieze with coefficients by plugging in the values of the $p_{ij}$.

Conversely, from any closed integral frieze with coefficients we obtain an integer point in the Grassmannian: the frieze relations are satisfied, and indeed, one can show that this already implies that all the Pl\"ucker relations are satisfied, see e.g. \cite[Theorem 3.3]{CuntzHolmJorgensen}.

\begin{Ex}[Grassmannian cluster algebras $Gr(k,n)$] \label{Ex:Grassmannian-kn}
The previous example can be generalized to $\cA(k,n):=\CC[Gr(k,n)]$ the homogeneous coordinate ring of the Grassmannian of $k$-subspaces of $\CC^n$, where we always assume that $k \leq \frac{n}{2}$. This is due to Scott \cite{Scott06}. We will only define the notions used later, for definitions and details about the non-explained terminology here, see \cite{Scott06,BaurIsfahan,Marsh13}. \\
In $\cA(k,n)$ one has Pl\"ucker coordinates $p_I$ indexed by $k$-subsets $I=\{i_1, \ldots, i_k\}$ of $[n]=\{1, \ldots, n\}$, where we assume $i_1 < \ldots <i_k$. The \emph{Pl\"ucker relations} for $Gr(k,n)$ are given by
\[ \sum_{r=0}^k (-1)^r p_{i_1, \ldots, i_{k-1}, j_r} p_{j_0, \ldots, \hat{j_r}, \ldots, j_k} \ , \]
where the sum is taken over all tuples $(i_1, \ldots, i_{k-1})$, $(j_0, \ldots, j_k)$ satisfying $1 \leq i_1< \cdots < i_{k-1} \leq n$ and $1 \leq j_0 < \cdots < j_k \leq n$.
Clusters consisting of Pl\"ucker coordinates are given by \emph{Postnikov diagrams}\footnote{But for $k \geq  3$ there also exist clusters containing higher degree cluster variables.}. One can also associate a quiver $Q_D$ containing frozen variables  to a Postnikov diagram $D$ and gets an isomorphism $\mc{A}(Q_D)\cong \cA(k,n)$ for each Postnikov diagram. \\
Further, the Pl\"ucker coordinates in a given Postnikov diagram satisfy the analogous properties as for a triangulation in the $Gr(2,n)$-case: they are a maximal set of \emph{non-crossing} Pl\"ucker coordinates. Here two Pl\"ucker coordinates $p_I$ and $p_J$ in $\cA(k,n)$, where $I$ and $J$ are $k$-subsets of $[n]$,  are called \emph{crossing} if there exist $i_1, i_2 \in I \backslash J$ and $j_1, j_2 \in J \backslash I$ with
$$i_1 < j_1 < i_2 < j_2 \text{ or } j_1 < i_1 < j_2 < i_2 \ . $$
Then $p_I$ and $p_J$ are non-crossing if they are not crossing. The size of a cluster is $(k-1)(n-k-1)$.
\end{Ex}

\subsection{Cluster categories - without coefficients} \label{Sub:cluster-cat-triang}

Now assume that our cluster algebra $\cA$ is of the form $\cA(\x , Q)$ (cf.~Remark \ref{Rk:Quiver_view}) and the  quiver $Q$ is associated to a simply laced Dynkin quiver \footnote{That means that the underlying graph of $Q$ is a Dynkin diagram of type $A_n, D_n, E_6, E_7$, or $E_8$. These quivers give in particular rise to representation finite algebras $KQ$, see e.g.~\cite{ARS-Book} for more details.}. This means that $\cA$ is a cluster algebra of finite cluster type, i.e., there are only finitely many cluster variables in $\cA$, see \cite{FZ2003II}.
If there are no frozen variables in $\cA$, that is $N=M$, then there is a categorification due to Buan, Marsh, Reineke, Reiten, and Todorov \cite{BMRRT06} via the path algebra of the quiver $Q$. Let us point out that this construction is surveyed in more detail in Pressland's lecture notes \cite{Pressland-survey}, also with connection to friezes. Moreover, see Keller's survey \cite[Section 6]{Keller-Iran}.

\begin{defi} \label{Def:cluster-cat}
Let $Q$ be a simply laced Dynkin quiver with any orientation. Denote by $KQ$ the path algebra of $Q$ (with coefficients in an algebraically closed field $K$) and let $D^b(Q):=D^b(\mmod{KQ})$ its derived module category. Then the \emph{cluster category} $\mc{C}_Q$ is defined to be the orbit category 
$$\mc{C}_Q=D^b(Q)/\tau^{-1}\! \circ \! [1] \ , $$
where $\tau$ is the Auslander--Reiten translation on $D^b(Q)$ and $[1]$ is the shift of the derived category. \\
An object $T \in \mc{C}_Q$ is \emph{rigid} if $\Hom(T,T[1])=0$ and $T$ is called \emph{maximally rigid} if $T=\sum_{i=1}^nT_i$ is rigid, the $T_i$ are indecomposable and pairwise non-isomorphic, and $n=|Q_0|$. An object $T$ is called \emph{cluster tilting} if 
\begin{align*}\add(T) & =\{ X \in \mc{C}_Q: \Ext^1_{\mc{C}_Q}(X,T')=0 \text{ for all } T' \in \add(T) \} \\
 & =\{ X \in \mc{C}_Q: \Ext^1_{\mc{C}_Q}(T',X)=0 \text{ for all } T' \in \add(T) \} \ .
  \end{align*}
Further, if $T$ is a cluster tilting object in $\mc{C}_Q$ then $B_T:=\End_{\mc{C}}(T)$ is called a \emph{cluster tilted algebra}. 
\end{defi}
Note here that $D^b(Q)$ can be viewed as $\bigcup_{i \in \ZZ}\mmod{KQ}[i]$ with connecting homomorphisms for any acyclic quiver $Q$ (that is, all objects in $D^b(Q)$ are modules or shifts of modules, see \cite[Section 4]{Happel87}). 

\begin{Rmk}
Since we will see in Section \ref{Sub:triangulation-quiver} that the quivers coming from triangulations of polygons may sometimes only be mutation equivalent to quivers of type $A$, one can also associate a cluster category in that case:
If $Q'$ is a quiver that is mutation equivalent to a Dynkin quiver, then one has to consider the generalized cluster category $\mc{C}_{Q',W}$, where $W$ is a potential, introduced by \cite{Amiot}. This category is derived equivalent to $\mc{C}_{Q}$.
\end{Rmk}

Let now $\mc{C}=\mc{C}_Q$ denote a cluster category as above. The cluster category $\mc{C}$ is triangulated and $2$-Calabi--Yau as shown in \cite{BMRRT06}. For more triangulated properties and details we refer to Keller's excellent survey \cite{Keller-quiver-triangulated}. In order to describe such a category most appropriately, one generally computes the \emph{Auslander--Reiten (=AR) quiver}. This is a quiver, whose vertices correspond to indecomposable objects in $\mc{C}$ and arrows correspond to irreducible maps between them. Moreover, the Auslander--Reiten translation is denoted by a dashed arrow. For more on AR-quivers, in particular in this context, see \cite{SchifflerBook}, or more generally in \cite{ASS, ARS-Book, Yoshino}.

In order to describe the AR-quiver, we use a derived equivalence between a quotient of $\mc{C}_Q$ and the module category of a cluster tilted algebra $B_T$, where $T$ is a cluster tilting object in $\mc{C}_Q$ (see Def.~\ref{Def:cluster-cat}). 

One can show that the module category $\mmod{B_T}$ is derived  equivalent to the quotient category $\mc{C} / \langle \add (T) [1]\rangle$, see \cite{BuanMarshReiten-TAMS} and \cite{KoenigZhu} for a more general framework.  This means that any indecomposable in $\mc{C}$ that is not isomorphic to the shift of a direct summand of  $T$ is an indecomposable module over $B_T$. The shifts of the direct summands $T_i$ of $T$ correspond to the  suspension of projective modules in the generalized cluster category of $B_T$.  Moreover, the structure of the Auslander--Reiten quivers is preserved, so that one can write down AR-sequences in $B_T$. One may now explicitly compute the AR-quiver and write down the modules. All of them are string modules and conversely, one can also obtain the modules from the AR-quiver: if $\mc{C}$ is a generalized cluster category of type $A_n$, then this is explained in Construction 2.24 \cite[Section 2]{BFGST18}.

\begin{Ex}
Let $Q$ be the $A_2$ quiver with $Q = \atwo$. Then the AR-quiver of $D^b(Q)$ looks as follows:
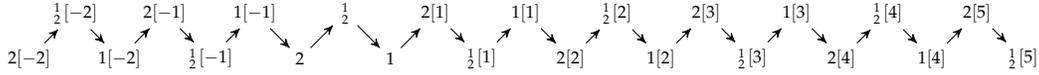
\[
\begin{tikzpicture}[scale=0.6, transform shape]

	\node   (Z2) at (2,10) {$\begin{smallmatrix}1\\2\end{smallmatrix}[-2]$};
	\node (Z3) at (4,10) {$2[-1]$};
	\node   (Z4) at (6,10) {$1[-1]$};
	\node (Z5) at (8,10) {$\begin{smallmatrix}1\\2\end{smallmatrix}$};
	\node  (Z6) at (10,10) {$2[1]$};
	\node (Z7) at (12,10) {$1[1]$};
	\node (Z8) at (14,10) {$\begin{smallmatrix}1\\2\end{smallmatrix}[2]$};
	\node   (Z9) at (16,10) {$2[3]$};
	\node (Z10) at (18,10) {$1[3]$};
	\node   (Z11) at (20,10) {$\begin{smallmatrix}1\\2\end{smallmatrix}[4]$};
	\node (Z12) at (22,10) {$2[5]$};
	
	\node (Y1) at (1,9) {$\cdots$};
	\node (Y2) at (3,9) {$1[-2]$};
	\node (Y3) at (5,9) {$\begin{smallmatrix}1\\2\end{smallmatrix}[-1]$};
	\node (Y4) at (7,9) {$2$};
	\node (Y5) at (9,9) {$1$};
	\node  (Y6) at (11,9) {$\begin{smallmatrix}1\\2\end{smallmatrix}[1]$};
	\node (Y7) at (13,9) {$2[2]$};
	\node (Y8) at (15,9) {$1[2]$};
	\node (Y9) at (17,9) {$\begin{smallmatrix}1\\2\end{smallmatrix}[3]$}; 
	\node (Y10) at (19,9) {$2[4]$};
	\node (Y11) at (21,9) {$1[4]$};
	\node (Y12) at (23,9) {$\cdots$};

	\draw[->] (Z2) -- (Y2);
	\draw[->] (Z3) -- (Y3);
	\draw[->] (Z4) -- (Y4);
	\draw[->] (Z5) -- (Y5);
	\draw[->] (Z6) -- (Y6);
	\draw[->] (Z7) -- (Y7);
	\draw[->] (Z8) -- (Y8);
	\draw[->] (Z9) -- (Y9);
	\draw[->] (Z10) -- (Y10);
	\draw[->] (Z11) -- (Y11);
	\draw[->] (Z12) -- (Y12);

	\draw[->] (Y1) -- (Z2);
	\draw[->] (Y2) -- (Z3);
	\draw[->] (Y3) -- (Z4);
	\draw[->] (Y4) -- (Z5);
	\draw[->] (Y5) -- (Z6);
	\draw[->] (Y6) -- (Z7);
	\draw[->] (Y7) -- (Z8);
	\draw[->] (Y8) -- (Z9);
	\draw[->] (Y9) -- (Z10);
	\draw[->] (Y10) -- (Z11);
	\draw[->] (Y11) -- (Z12);

\end{tikzpicture}
\]
The AR-quiver of the quotient $\mc{C}_Q$ becomes periodic
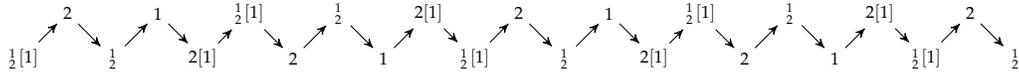
\[
\begin{tikzpicture}[scale=0.6, transform shape]

	\node   (Z2) at (2,10) {$2$};
	\node (Z3) at (4,10) {$1$};
	\node   (Z4) at (6,10) {$\begin{smallmatrix}1\\2\end{smallmatrix}[1]$};
	\node (Z5) at (8,10) {$\begin{smallmatrix}1\\2\end{smallmatrix}$};
	\node  (Z6) at (10,10) {$2[1]$};
	\node (Z7) at (12,10) {$2$};
	\node (Z8) at (14,10) {$1$};
	\node   (Z9) at (16,10) {$\begin{smallmatrix}1\\2\end{smallmatrix}[1]$};
	\node (Z10) at (18,10) {$\begin{smallmatrix}1\\2\end{smallmatrix}$};
	\node   (Z11) at (20,10) {$2[1]$};
	\node (Z12) at (22,10) {$2$};
	
	\node (Y1) at (1,9) {$\cdots$};
	\node (Y2) at (3,9) {$\begin{smallmatrix}1\\2\end{smallmatrix}$};
	\node (Y3) at (5,9) {$2[1]$};
	\node (Y4) at (7,9) {$2$};
	\node (Y5) at (9,9) {$1$};
	\node  (Y6) at (11,9) {$\begin{smallmatrix}1\\2\end{smallmatrix}[1]$};
	\node (Y7) at (13,9) {$\begin{smallmatrix}1\\2\end{smallmatrix}$};
	\node (Y8) at (15,9) {$2[1]$};
	\node (Y9) at (17,9) {$2$}; 
	\node (Y10) at (19,9) {$1$};
	\node (Y11) at (21,9) {$\begin{smallmatrix}1\\2\end{smallmatrix}[1]$};
	\node (Y12) at (23,9) {$\cdots$};
	
	\draw[->] (Z2) -- (Y2);
	\draw[->] (Z3) -- (Y3);
	\draw[->] (Z4) -- (Y4);
	\draw[->] (Z5) -- (Y5);
	\draw[->] (Z6) -- (Y6);
	\draw[->] (Z7) -- (Y7);
	\draw[->] (Z8) -- (Y8);
	\draw[->] (Z9) -- (Y9);
	\draw[->] (Z10) -- (Y10);
	\draw[->] (Z11) -- (Y11);
	\draw[->] (Z12) -- (Y12);

	\draw[->] (Y1) -- (Z2);
	\draw[->] (Y2) -- (Z3);
	\draw[->] (Y3) -- (Z4);
	\draw[->] (Y4) -- (Z5);
	\draw[->] (Y5) -- (Z6);
	\draw[->] (Y6) -- (Z7);
	\draw[->] (Y7) -- (Z8);
	\draw[->] (Y8) -- (Z9);
	\draw[->] (Y9) -- (Z10);
	\draw[->] (Y10) -- (Z11);
	\draw[->] (Y11) -- (Z12);

\end{tikzpicture}
\]
Here a cluster tilting object would be $T=2 \oplus \begin{smallmatrix}1\\2\end{smallmatrix}$. 
\end{Ex}

In a cluster category $\mc{C}_Q$, where $Q$ is a Dynkin quiver, there is also a mutation procedure with respect to a chosen cluster tilting object $T$, which corresponds to mutation in a cluster algebra of the same Dynkin type as $Q$.  More precisely, mutation of a cluster tilting object in the cluster category $\mc{C}_Q$ corresponds to the mutation of the quiver of the endomorphism algebra of the corresponding cluster tilting object. In Section \ref{Sub:cluster-char} we will see that one can recover the cluster variables in the algebra $\mc{A}_Q$ (a cluster algebra of finite cluster type with trivial coefficients) from $\mc{C}_Q$.

\begin{defi}[cf.~\cite{BMRRT06}] Let $T=\bigoplus_{i=1}^nT_i$ be a cluster tilting object in a cluster category $\mc{C}_Q$. A \emph{mutation
of $T$ in direction $i$} is a new cluster tilting object $\mu_i(T)= T/T_i \oplus T_i'$ where $T_i'$ is defined as the pseudokernel of the right $\add(T/T_i)$-approximation of $T_i$ or pseudocokernel of the left $\add(T/T_i)$-approximation of $T_i$:
\[ \xrightarrow{}T_i' \xrightarrow{f_i} B \xrightarrow{} T_i \xrightarrow{} \text{ or } \xrightarrow{} T_i \xrightarrow{} B \xrightarrow{f_i} T_i' \xrightarrow{} \ .  \]
\end{defi}

\subsubsection{Triangulations and cluster categories of type $A$} \label{Sub:triangulation-quiver}
We conclude this section with giving an account how to move from a triangulation of an $n$-gon  $P_n$ to a cluster category $\mc{C}_Q$ of type $A_{n-3}$. Combining this with the results in Section \ref{Sec:ClusterAlgs} we get a correspondence between cluster algebras of type $A$, cluster categories of type $A$, triangulations of polygons, and CC-friezes. 
We follow the approach of \cite{CCS06}, also discussed in \cite[Section 3.1.3 and 3.4]{SchifflerBook}. We will discuss in Section \ref{Sub:cluster-char} how to construct a cluster character map directly from $\mc{C}_Q$ to $\mc{A}_Q$ and via specialization to the corresponding CC-frieze.

First assume that $P_n$ is a triangulated $n$-gon with triangulation $T$. Here we omit the boundary arcs, that is, $T$ consists of $n-3$ non-crossing diagonals in $P_n$. Number the elements of $T$ by $1, \ldots, n-3$ - these will be the vertices of the quiver $Q_T$ associated to $T$. There are two types of triangles in the triangulation: (i) either a triangle consists of two elements $i$, $j$ of $T$ and a boundary edge or (ii) a triangle consists of three elements $i,j,k$ of $T$. In case (i) we draw an arrow from $i$ to $j$ if the orientation is as in Fig.~\ref{Fig:triang-quiver-arrow}. In case (ii) we draw a $3$-cycle of arrows between $i$, $j$, and $k$ as in Fig.~\ref{Fig:triang-quiver-arrow}.  Thus we obtain the quiver $Q_T$.\\
\begin{figure}[h!]
\begin{minipage}{0.45 \textwidth}
\begin{tikzpicture}[scale=2]

\foreach \i in {1,...,6}{
  \coordinate (A\i) at (60*\i:1);
}

\fill[lightgray] (A1)--(A4)--(A5)--cycle;
\draw[thick] (A1)--(A2)--(A3)--(A4)--(A5)--(A6)--cycle;

\draw[thick] (A1)--(A3);
\draw[thick] (A1)--(A4);
\draw[thick] (A1)--(A5);

    \fill (A1) circle (1pt) node[shift={(60:0.25)}] {$a$};
     \fill (A5) circle (1pt) node[shift={(120:-0.25)}] {$c$};
   \fill (A4) circle (1pt) node[shift={(180:0.3)}] {$b$};
   \fill (A3) circle (1pt) node[shift={(240:0.3)}] {};

\node at ($(A1)!0.5!(A5)$) [right=1pt] {$i$};
\node at ($(A1)!0.5!(A4)$) [left=1pt] {$j$};

\coordinate (M_ac) at ($(A1)!0.5!(A5)$);
\coordinate (M_ab) at ($(A1)!0.5!(A4)$);
\draw[->, thick, red] (M_ac) -- (M_ab);

\end{tikzpicture}
\end{minipage}
\begin{minipage}{0.45 \textwidth}
\begin{tikzpicture}[scale=2]

\foreach \i in {1,...,6}{
  \coordinate (A\i) at (60*\i:1);
}

\fill[lightgray] (A1)--(A3)--(A5)--cycle;
\draw[thick] (A1)--(A2)--(A3)--(A4)--(A5)--(A6)--cycle;

\draw[thick] (A1)--(A3);
\draw[thick] (A1)--(A5);
\draw[thick] (A3)--(A5);

    \fill (A1) circle (1pt) node[shift={(60:0.25)}] {$a$};
     \fill (A5) circle (1pt) node[shift={(120:-0.25)}] {$c$};
   \fill (A4) circle (1pt) node[shift={(180:0.3)}] { };
   \fill (A3) circle (1pt) node[shift={(240:0.3)}] {$d$};

\node at ($(A1)!0.5!(A5)$) [right=1pt] {$i$};
\node at ($(A1)!0.5!(A3)$) [above left=-1pt] {$j$};
\node at ($(A3)!0.5!(A5)$) [below left=-1pt] {$k$};

\coordinate (M_ac) at ($(A1)!0.5!(A5)$);
\coordinate (M_ad) at ($(A1)!0.5!(A3)$);
\coordinate (M_cd) at ($(A3)!0.5!(A5)$);

\draw[->, thick, red] (M_ac) -- (M_ad);
\draw[->, thick, red] (M_ad) -- (M_cd);
\draw[->, thick, red] (M_cd) -- (M_ac);

\end{tikzpicture}
\end{minipage}
\caption{Arrows in the \rot{quiver}: case (i) left, case (ii) right.}
\label{Fig:triang-quiver-arrow}
\end{figure}
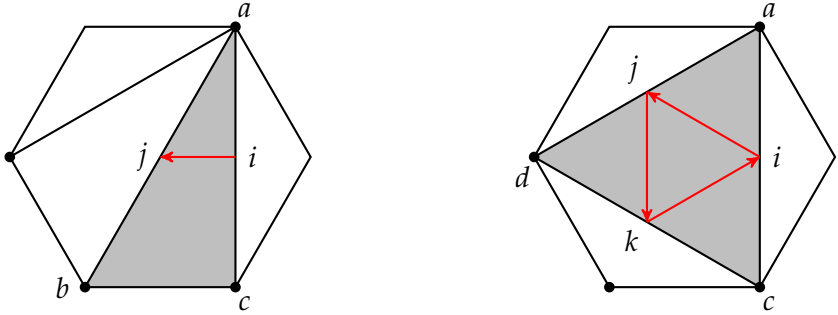 

Now the cluster tilted algebra $B_T=KQ_T/I$ is given as the path algebra of the quiver $Q_T$ modulo the relations coming from $3$-cycles in $Q_T$: whenever $i \xrightarrow{\alpha} j \xrightarrow{\beta} k \xrightarrow{\gamma} i$ is a $3$-cycle, then $\alpha \beta , \beta \gamma $, and $\gamma \alpha $ are in $I$. 
Note that $B_T$ is isomorphic as a ring to $\End_{\mc{C}}(T)$ for $T$ the cluster tilting object in $\mc{C}$ the cluster category corresponding to the triangulation $T$. The AR-quiver for $B_T$ can be constructed by the usual process of knitting, as e.g. explained in \cite[Section 3]{SchifflerBook}. One can also construct the AR-quiver of $B_T$ directly from the triangulation $T$, as explained in detail in \cite[Sections 3.1.3 and 3.4.1]{SchifflerBook}.

The procedure can also be reversed: if $\mc{C}$ is a cluster category of type $A$, then using the AR-quiver of $B_T$ one can reconstruct the quiver $Q$. From $Q$ it is easy to obtain the triangulated polygon.

\subsection{Grassmannian cluster categories}

We have already introduced Grassmannian with its cluster algebra $\cA(k,n)=\CC[Gr(k,n)]$ in Example \ref{Ex:Grassmannian-kn}. There is also a Frobenius categorification of $\cA(k,n)$ due to Jensen, King, and Su \cite{JKS16}, inspired by the categorification of an affine open cell in the Grassmannian by \cite{GLS-flag}. The JKS category is constructed from a plane curve singularity $\{x^k-y^{n-k}=0\} \subseteq \CC^2$, which will be the key idea to generalize to the infinite Grassmannian $Gr(k,\infty)$, see Section \ref{Sub:infinite}. 

Let $C_n=\langle \zeta \in \CC^*: \zeta^n=1\rangle$\footnote{This group is sometimes denoted by $\mu_n$ in the literature, but we use a different notation to avoid confusion with the mutation $\mu_n$.}, where $\zeta$ is a primitive root of unity, act on $\CC^2$ via the usual action
$$(x,y) \mapsto (\zeta x, \zeta^{-1} y) \ . $$
The element $f=x^k - y^{n-k}$ is a semi-invariant under the action of $C_n$, however, the $C_n$-action restricts to the plane curve singularity with coordinate ring $R_{k,n}=\CC\llbracket x,y \rrbracket / \langle x^k-y^{n-k}\rangle$. The invariant subring of the $C_n$-action is $Z=\CC\llbracket t \rrbracket$ with $t=xy$. \\
Then look at the category of $C_n$-equivariant finitely generated modules over $R_{k,n}$, namely $\textrm{mod}_{C_n}(R_{k,n})$, which is tautologically equivalent to the category of finitely generated modules over the skew-group ring $C_n * R_{k,n}$. The \emph{Grassmannian cluster category of type $(k,n)$} is defined to be
$$\mc{C}(k,n)=\CM (C_n * R_{k,n})=\CM_{C_n}(R_{k,n}) \ , $$
that is, the category of maximal Cohen--Macaulay modules over $C_n * R_{k,n}$. In particular, the  objects of $\mc{C}(k,n)$ are modules over the skew-group ring that are free over the invariant subring $Z$. 

\begin{Rmk}
Alternatively, one can describe the category using quivers: let $B=B_{k,n}$ be the completion of the path algebra 
$\CC Q(n)/\langle xy-yx,x^k-y^{n-k}\rangle$, 
where $Q(n)$ is a quiver with $n$ vertices $1, \ldots, n$ ordered in a circle and there are arrows $i \xrightarrow{x_{i+1}} i+1$, $i \xrightarrow{y_i} i-1$ for all $i$\footnote{This is the double quiver $\bar Q$ of $Q$ of extended Dynkin type $A_{n-1}$.}, $xy-yx$ stands for the $n$ relations $y_ix_i - x_{i+1}y_{i+1}$, $i=1,\dots,n$ and $x^k-y^{n-k}$ stands for the $n$ relations $x_{i+1} \cdots x_{i+k} -  y_iy_{i-1} \cdots  y_{i-n+k+1}$ 
(reducing indices modulo $n$). Then $\mc{C}(k,n)=\CM(B)$. We refer to \cite{JKS16, BaurIsfahan} for details on this construction.
\end{Rmk}

The category  $\mc{C}(k,n)$ is a Frobenius exact category, is Krull--Schmidt, and it is stably 2-CY. Note that the stable category $\underline{\mathcal C(k,n)}$ is triangulated, which follows from \cite[Theorem 4.4.1]{BuchweitzMCM}, since $B$ is Iwanaga--Gorenstein, or alternatively from \cite[I, Theorem 2.6]{Happel88}, since $\mc{C}(k,n)$ is a Frobenius category. The category $\mc{C}(k,n)$ is an additive categorification of the cluster algebra structure of $\mathcal A(k,n)$, see \cite{JKS16}. 
Moreover, $\mc{C}(k,n)$ has  AR-sequences \cite[Remark 3.3]{JKS16} and in particular the AR-quivers for the finite type cluster categories, that is $\mc{C}(k,n)$ for $k=2, n \geq 4$ or $k=3, n \in \{6,7,8\}$,  have been described in loc.~cit., also see \cite[Section 7]{FMP} for a detailed description of the maps for $\mc{C}(2,n)$.

In view of the relation with friezes, we also collect the following:

\begin{fact} \label{Fact:JKS-cat}
\begin{enumerate}[(1)]
\item As already mentioned above, $\mc{C}(k,n)$ is of \emph{finite type}, if it has finitely many isomorphism classes of indecomposable modules. This happens for $k \leq \frac{n}{2}$ if and only if 
either $k=2$ and $n$ arbitrary, or $k=3$ and $n\in \{6,7,8\}$. These categories are of Dynkin type $A_{n-3}$, $D_4$, $E_6$ and $E_8$ respectively. This means that one can find a cluster-tilting object in the stable category $\underline{\mc{C}(k,n)}$ whose quiver 
is of the corresponding Dynkin type.
\item The \emph{rank} of an object $X \in\mc{C}(k,n)$ is defined to be the length of $X \otimes_ZK$ for $K$ the field of fractions of the invariant subring $Z$, see \cite[Definition 3.5]{JKS16}. 
There is a bijection between the rank one indecomposable objects of $\mc{C}(k,n)$ and $k$-subsets of $[n]$ (here we denote by $[n]$ the interval of integers $1, 2, \ldots, n$). Thus any rank one module in $\mc{C}(k,n)$ is of the form $M_I$ where $I$ is a $k$-subset of $[n]$. In particular, the indecomposable projective-injective objects are indexed by the $n$ consecutive $k$-subsets of $[n]$ (where we consider cyclic consecutive $k$-subsets modulo $n$).
The rank one indecomposables of $\mc{C}(k,n)$ are thus in bijection with the cluster variables of 
$\mathcal A(k,n)$ which are Pl\"ucker coordinates, cf.~Example \ref{Ex:Grassmannian-kn}.  For $\mc{A}(2,n)$ this means that we have a bijection between Pl\"ucker coordinates and indecomposable objects in $\mc{C}(k,n)$.
\item The number of indecomposable non projective-injective summands in any cluster-tilting object of $\mc{C}(k,n)$ 
is $(k-1)(n-k-1)$. This is called the {\em rank} of the cluster category $\mc{C}(k,n)$. \\
For $k=2$, this means that a cluster tilting object $T$ in $\mc{C}(2,n)$ has in total $(n-3) + n=2n-3$ direct summands, which correspond to the diagonals in the triangulation of $P_n$ corresponding to $T$ plus the boundary edges $(i,i+1)$ of the polygon $P_n$.
\end{enumerate}
\end{fact}

\begin{Rmk}
The category $\mc{C}(k,n)$ can also be combinatorially studied with the help of dimer models, see \cite{BKM16}. There is also research on understanding the higher rank modules and infinite type categories and connections to root combinatorics, we refer to \cite{BBGE18} for more details. \end{Rmk}

\begin{Ex} 
Consider $\mc{C}(2,n)=$ $\CM_{C_n}(\CC\llbracket x,y \rrbracket / \langle x^2-y^{n-2} \rangle)$. This category consists of $n \choose 2$ indecomposable objects $M_{ab}$ with $1 \leq a < b \leq n$, where the projective-injectives are $M_{i,i+1}$. The AR-quiver can be visualized as follows in Fig.~\ref{Fig:ARquiver2n} (note here that the quiver continues in both horizontal directions with entries as indicated - in particular there are only finitely many non-isomorphic indecomposable objects):
\begin{figure}[h!]
\begin{tikzpicture}
\draw (0,6)  node (M1n)   {$M_{1,2}$};
\draw (2,6)  node (Mn-1n) {$M_{2,3}$};
\draw (4,6)  node (Mn-2n-1) {$M_{3,4}$};
\draw (10,6) node (M23)   {$M_{n-1,n}$};
\draw (12,6) node (M12r)  {$M_{1,n}$};
\draw (1,5)  node (M1n-1) {$M_{1,3}$};
\draw (3,5)  node (Mn-2n) {$M_{2,4}$};
\draw (9,5)  node (M24)   {$M_{n-2,n}$};
\draw (11,5) node (M13r)  {$M_{n-1,1}$};
\draw (2,4)  node (M1n-2) {$M_{1,4}$};
\draw (8,4)  node (M25)   {$M_{n-3,n}$};
\draw (10,4) node (M14r)  {$M_{n-2,1}$};
\draw (4,2)  node (M13)   {$M_{1,n-1}$};
\draw (6,2)  node (M2n)   {$M_{2,n}$};
\draw (8,2)  node (M1n-1r){$M_{1,3}$};
\draw (5,1)  node (M12)   {$M_{1,n}$};
\draw (7,1)  node (M1nr)  {$M_{1,2}$};

\foreach \s/\t in {1n/1n-1, 1n-1/n-1n, 1n-1/1n-2, n-1n/n-2n, 1n-2/n-2n, n-2n/n-2n-1,
                   23/13r, 13r/12r, 24/23, 24/14r, 14r/13r, 25/24, 13/12, 12/2n, 2n/1nr, 1nr/1n-1r}
{\path [->] (M\s) edge (M\t);}

\foreach \s/\t in {1n-1/n-2n, 24/13r, 25/14r, 13/2n, 2n/1n-1r}
{\path [dashed] (M\s) edge (M\t);}

\foreach \x/\y/\a/\b in {6/6/n-3/n-2, 8/6/3/4, 5/5/n-3/n-1, 7/5/3/5,
                         4/4/n-3/n, 3/3/1/n-3, 5/3/n-4/n, 7/3/2/6, 9/3/1/5}
{\draw (\x,\y) node (M\a\b) {};}

\foreach \s/\t in {n-2n-1/n-3n-1, n-2n/n-3n, 1n-2/1n-3, 34/24, 35/25, 26/25, 15/14r}
{\path [->] (M\s) edge (M\t);}

\foreach \s/\t in {1n-3/13, n-3n/n-4n, 2n/26, 1n-1r/15, n-2n-1/34, n-3n-1/35}
{\path [dotted] (M\s) edge (M\t);}
\end{tikzpicture}

\caption{The Auslander--Reiten quiver of $\mc{C}(2,n)$.}
\label{Fig:ARquiver2n}
\end{figure}
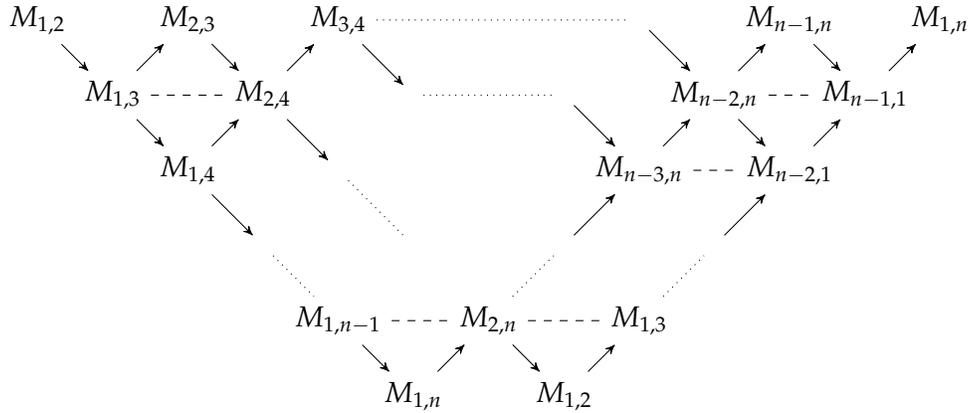
Here note that the AR-translation (dashed edges) is only defined for non-projective-injective objects. Now choose a triangulation $T$ of an $n$-gon $P_n$. Using the bijection between triangulations of $n$-gons and CC-friezes from Theorem \ref{Thm:CoCo} one  can easily see that setting each $M_{i,i+1}=1$ (cyclically modulo $n$) and each $M_{ab}$ for $(a,b) \in T$ equal to $1$ will yield a CC-frieze. In the next Section we will explain that this also follows from properties of the cluster character which associates to each indecomposable object of $\mc{C}(k,n)$ an element in $\mc{A}(k,n)$.
\end{Ex}

\begin{Ex}[AR-quivers for the other finite type $\mc{C}(k,n)$] The AR-quivers for $\mc{C}(3,n)$ for $n \in \{6,7,8 \}$ are of type $D_4, E_6, E_8$, respectively (plus some additional projectives) and are described in \cite{JKS16}. For infinite type Grassmannian cluster categories there are some partial results, see \cite{BaurIsfahan} for an overview.
\end{Ex}

\section{Friezes from representation theory} \label{Sec:friezes-cat}

In this section we first consider cluster characters and then show how to obtain several of the frieze types introduced in Section \ref{Sec:friezes} from them. In particular, we have a correspondence between the Frobenius categories $\mc{C}(2,n)$, the cluster structure of the Grassmannian $\mc{A}(2,n)$, triangulations of an $n$-gon $P_n$, and CC-friezes of width $n-3$, see the overview Table in Fig.~\ref{Fig:Correspondence}.

For Grassmannians of higher rank, i.e., $k \geq 3$, one still obtains tame $SL_k$-friezes, via Pl\"ucker friezes, see Section \ref{Sec:Pluecker}. Finally we will discuss Grassmannian cluster categories of infinite rank and frieze structures obtained from them in Section \ref{Sub:infinite}.

\subsection{Cluster character and friezes} \label{Sub:cluster-char}

In the previous sections it was described how to move between a triangulation of an $n$-gon, the cluster algebra $\mc{A}(2,n)$ (resp.~$\mc{A}_Q$, where $Q$ is a quiver mutation equivalent to a Dynkin quiver of type $A_{n-3}$), the cluster category $\mc{C}(2,n)$ (resp.~ $\mc{C}_Q$) and CC-friezes. Here we describe the map from the category to the frieze: after choosing a cluster tilting object $T$, the cluster character  associates to each indecomposable object in the category a Laurent polynomial in the initial cluster corresponding to $T$ in the cluster algebra and it satisfies various multiplication formulas. 

The \emph{Caldero--Chapoton map} (or \emph{cluster character}) was first considered by Caldero and Chapoton for cluster  categories \cite{CalderoChapoton}, in particular, they noted the connection to friezes (in the case where the associated quiver is a type $A_{n}$ Dynkin quiver). Then it has been extended to more general settings over the last two decades, notably by Palu \cite{Palu} for $2$-Calabi--Yau triangulated categories with cluster tilting objects, and Fu and Keller for Frobenius exact categories. A version for extriangulated categories is due to Wang, Wei, and Zhang \cite{WWZ2}. See \cite{Plamondon-CIMPA} for an introduction to cluster characters.

The definition of a cluster character is due to Palu \cite{Palu}, we need it for two cases: the cluster category or generalized cluster category $\mc{C}$ and the Frobenius exact category $\mc{C}(k,n)$. In order to keep the presentation compact we state the definition for extriangulated Frobenius categories, which encompasses both cases (cf.~\cite[Definition 4.1]{WWZ2}).

\begin{defi} \label{Def:CC}
Let $\mc{F}$ be a $2$-Calabi--Yau extriangulated category. A \emph{cluster character} on $\mc{F}$ with values in a commutative ring $R$ is a map $\cc: \rm{obj} (\mc{F}) \xrightarrow{}$ $R$ satisfying
\begin{enumerate}[(a)]
\item If $M \cong N$, then $\cc(M)=\cc(N)$;
\item\label{CC:Product} For any objects $M,N \in \mc{F}$ it holds that $\cc(M \oplus N)=\cc(M) \cdot \cc(N)$;
\item \label{Exchange-seq} For any objects $M,N \in \mc{F}$ if $\dim_K \mathbb{E}(N,M)=1$ and
\begin{equation} \label{Eq:exchange-seq}
 \begin{tikzcd}N \arrow{r}&  L \arrow{r} & M \arrow[dashed]{r} & \phantom{l} \end{tikzcd} \text{ and \ } \begin{tikzcd}M \arrow{r}&  L' \arrow{r} & N \arrow[dashed]{r} & \phantom{l} \end{tikzcd}
 \end{equation}
are nonsplit $\mathbb{E}$-triangles, then $\cc(M) \cdot \cc(N)=\cc(L) + \cc(L')$.
\end{enumerate}
\end{defi}

\begin{Rmk}
The $K$-bilinear functor $\mathbb{E}: \mc{F} \times \mc{F} \xrightarrow{} \mmod K$, where $K$ is a field, in \eqref{Exchange-seq} above is part of the definition of an extriangulated category, see \cite{NakaokaPalu}. In the case where $\mc{F}$ is an exact category, $\mathbb{E}(M,N)=\Ext^1_{\mc{F}}(M,N)$ for any objects $M,N \in \mc{F}$, and in the case where $\mc{F}$ is a triangulated $2$-CY category $\mathbb{E}(M,N)=\Hom_{\mc{F}}(M,N[1])$.
Moreover the conflations in \eqref{Eq:exchange-seq} are nonsplit exact sequences
\[  \begin{tikzcd} 0 \arrow{r} &N \arrow{r}&  L \arrow{r} & M \arrow{r} & 0\end{tikzcd} \text{ and } \begin{tikzcd}0 \arrow{r} & M \arrow{r}&  L' \arrow{r} & N \arrow{r} & 0 \end{tikzcd} \]
in the exact case, and nonsplit triangles
\[ \begin{tikzcd}N \arrow{r}&  L \arrow{r} & M \arrow{r} & N[1] \end{tikzcd} \text{ and } \begin{tikzcd}M \arrow{r}&  L' \arrow{r} & N \arrow{r} & M[1] \end{tikzcd} 
\]
in the triangulated case. 
\end{Rmk}

One can give an explicit formula for the cluster character in terms of quiver Grassmannians of submodules in both cases, however, since we are interested in friezes here, we will omit these expressions and refer the interested reader to the references cited above.

However, we note the following:

\begin{fact} \label{Fact:CC}
\begin{enumerate}[(1)]
\item Let $\mc{C}$ be a (generalized) cluster category of type $A_n$ with cluster tilting object $T=\bigoplus_{i=1}^n T_i$, let $B_T$ the corresponding cluster tilted algebra, and let $\cc_T: \mc{C} \xrightarrow{} \QQ(x_1, \ldots, x_n)$ be a cluster character. Then $\cc_T(T_i[1])=x_i$ and $\cc_T(M)$ is a Laurent polynomial with positive coefficients in the $x_i$ for any $M \in \mmod{B_T}$ (because each indecomposable is sent to a cluster variable, see e.g. \cite[Prop.~2.3]{FuKeller}, and by the positivity phenomenon, each cluster variable is a Laurent polynomial with positive coefficients in the initial cluster).
\item \label{Fact:CC-Frobenius} A cluster character exists for $\mc{C}(k,n)$, see \cite{JKS16}. This cluster character, also denoted by $\cc_T$ by abuse of notation, satisfies the properties of the Fu--Keller cluster character. In particular, assume that $T=\bigoplus_{i} T_i$ is a cluster tilting object with pairwise nonisomorphic indecomposable direct summands $T_i$ corresponding to Pl\"ucker coordinates $p_I$, that is, each index $i$ corresponds to a $k$-subset $I \subseteq [n]$ (such a cluster exists by \cite{Scott06}). Then $\cc_T(T_i)=p_I$ and $\cc_T(M)$ is a Laurent polynomial with positive coefficients in the $p_I$ for any  indecomposable $M \in \mc{C}(k,n)$ that is reachable from $T$ (same argument as above, but using \cite[Thm.~3.3]{FuKeller}).
\item For $\mc{C}(2,n)$ with a cluster tilting object $T=\bigoplus_{l=1}^{n-3}T_l \oplus \bigoplus_{l=1}^n M_{l,l+1}$, where each $T_l$ corresponds to a rigid rank $1$ module $M_{i_l,j_l}$,  we have 
$$\cc_T: \mc{C}(2,n) \xrightarrow{} \QQ(p_{i_1,j_1}, \ldots, p_{i_{n-3},j_{n-3}}, p_{1,2}, \ldots, p_{n-1,n},p_{1,n})$$ with $\cc_T(T_l)=p_{i_l,j_l}$ and $\cc_T(M_{l,l+1})=p_{l,l+1}$. 
\end{enumerate}
\end{fact}

\begin{Lem} \label{Lem:Pluecker-frieze-coeff}
Consider $\mc{C}(2,n)$ with a cluster tilting object $T=\bigoplus_{l=1}^{n-3} M_{i_l,j_l} \oplus \bigoplus_{l=1}^n M_{l,l+1}$ and corresponding cluster character $\cc_T: \mc{C}(2,n) \xrightarrow{} \QQ(p_{i_1,j_1}, \ldots, p_{i_{n-3},j_{n-3}}, p_{12}, \ldots, p_{1n})$ and look at the AR-quiver of $\mc{C}(2,n)$. Apply the cluster character to each indecomposable module, forget the arrows and add each  a row of $0$'s on top and on the bottom. \\
Then the entries of rows 2 to $n-2$ of the resulting array are all non-zero in \\ $\QQ(p_{i_1,j_1}, \ldots, p_{i_{n-3},j_{n-3}}, p_{12}, \ldots, p_{1n})$ and satisfy the relations for a frieze with coefficients with $c_{i,i+1}=p_{i,i+1}$. \end{Lem}

\begin{proof}
W.l.o.g. assume that all our modules are in the fundamental domain of the AR-quiver, depicted in Fig.~\ref{Fig:ARquiver2n}. The AR-sequences are of the form
\[ 0 \xrightarrow{} M_{i,j} \xrightarrow{} M_{i,j+1} \oplus M_{i+1,j} \xrightarrow{} M_{i+1,j+1} \xrightarrow{} 0 \ \]
for $|i-j|>1$. Note that there are no AR-sequences starting or ending at the projective-injective modules $M_{i,i+1}$. By definition, each AR-sequence is a nonsplit short exact sequence with $\dim_{\CC}\Ext^1_{\mc{C}(2,n)}(M_{i+1,j+1}, M_{ij})=1$.
The corresponding sequence with start and endterm reversed is
\[ 0 \xrightarrow{} M_{i+1,j+1} \xrightarrow{} M_{i,i+1} \oplus M_{j,j+1} \xrightarrow{} M_{i,j} \xrightarrow{} 0 \ .\]
Applying product and sum formula \eqref{CC:Product} and \eqref{Exchange-seq} from Definition \ref{Def:CC}, we obtain the  frieze relation from Definition \ref{Def:frieze-coeffs}:
\begin{align*} \cc_T(M_{i,j}) \cc_T(M_{i+1,j+1}) & = \cc_T(M_{i,j+1}) \cc_T(M_{i+1,j})+\cc_T(M_{i,i+1}) \cc_T(M_{j,j+1}) \\
& = \cc_T(M_{i,j+1}) \cc_T(M_{i+1,j}) +p_{i,i+1} p_{j,j+1} \ .
\end{align*}
All $\cc_T(M_{ij})$ are nonzero Laurent polynomials in $\QQ(p_{i_1,j_1}, \ldots, p_{i_{n-3},j_{n-3}}, p_{12}, \ldots, p_{1n})$, see Fact \ref{Fact:CC}. Further from the periodicity of the AR-quiver, the periodicity assumptions on the entries of a frieze with coefficients from Definition \ref{Def:frieze-coeffs} are satisfied. Hence the claim follows.
\end{proof}

\begin{Prop} \label{Prop:frieze-from-CC}
There is a $1-1$-correspondence between cluster tilting objects in $\mc{C}(2,n)$ and CC-friezes of width $n-3$. In particular, each such CC-frieze is obtained by applying the cluster character to the entries of the AR-quiver of $\mc{C}(2,n)$ (and forgetting the arrows in the quiver plus adding a top and bottom row of $0$'s to each frieze). 
\end{Prop}

\begin{proof}
We already know the correspondence between cluster tilting objects in $\mc{C}(2,n)$ and triangulations of an $n$-gon (see Example \ref{Ex:Gr(2,n)}). Combined with the bijection of triangulations of an $n$-gon and CC-friezes of width $n-3$ from the Conway--Coxeter Theorem \ref{Thm:CoCo} the first statement follows. \\
Let now $T$ be a cluster tilting object for $\mc{C}(2,n)$. From Lemma \ref{Lem:Pluecker-frieze-coeff} we get a frieze with coefficients and entries Laurent series in the Pl\"ucker coordinates $p_{i_1,j_1}, \ldots, p_{i_{n-3},j_{n-3}}, p_{12}, \ldots, p_{1n}$ corresponding to the arcs in the triangulation $T$. Specializing all these $p_{ij}=1$ shows that the second and penultimate row consist only of $1$'s and the frieze relation \eqref{Eq:diamond} is satisfied by all entries. Since all entries are Laurent series with positive coefficients in the $p_{i_1,j_1}, \ldots, p_{i_{n-3},j_{n-3}}, p_{12}, \ldots, p_{1n}$ (see Fact \ref{Fact:CC} \eqref{Fact:CC-Frobenius}), it follows that they are all positive integers. Hence all properties for a CC-frieze are satisfied. \\
Conversely, picking a maximal set of $1$'s in the fundamental domain of a CC-frieze of width $n-3$ corresponds to a triangulation of an $n$-gon, cf.~Remark \ref{Rmk:1-in-frieze}, and hence to a cluster tilting object in $\mc{C}(2,n)$.
\end{proof}

 We have shown that in the Frobenius case we obtain all CC-friezes via applying the cluster character w.r.t. a cluster tilting module $T \in \mc{C}(2,n)$ to the AR-quiver of $\mc{C}(2,n)$. However, Caldero and Chapoton already observed  in their foundational paper \cite[Section 5]{CalderoChapoton} the corresponding statement for the cluster category of type $A$. We briefly explain the construction for the generalized cluster category of type $A_{n-3}$: 
 
Let $\mc{C} = \mc{C}_{(Q,W)}$ be a generalized cluster category of type $A_{n-3}$ and let $T = \bigoplus_{i=1}^{n-3} T_i$ be a cluster  tilting object of $\mc{C}$, with pairwise non-isomorphic direct summands $T_i$. Let $B_T = \End_\mc{C} (T)$ be the cluster-tilted algebra associated to $T$. Then each indecomposable object in $\mc{C}$ is either an indecomposable $B_T$-module or $T_i[1]$, a shift of a direct summand of the cluster tilting object. The \emph{specialized Caldero Chapoton map} is the map obtained from postcomposing the Caldero Chapoton map $\cc_T$ (explicit formula in \cite[Section 3]{CalderoChapoton}) associated to $T$ with the specialization of the initial cluster variables to $1$:
\[
 \cc_T(M) = \begin{cases}
              1 & \text{ if } M = T_i[1] \\
              \sum_{\underline{e}}\chi(Gr_{\underline{e}}(M)) & \text{ if $M$ is a $B_T$-module}.
             \end{cases}
\]
Here, $Gr_{\underline{e}}(M)$ denotes the Grassmannian of submodules of the $B_T$-module $M$ with dimension vector $\underline{e}$ and $\chi$ is the Euler-Poincar\'e characteristic.
If $\mc{C}$ is of type $A_{n-3}$ and $T$ is a cluster-tilting object of $\mc{C}$, then for every indecomposable $B_T$-module $M$ the Grassmannian $Gr_{\underline{e}}(M)$ is either empty or a point, 
therefore the above formula simplifies to
\[
 \cc_T(M) = \sum_{N \subseteq M} 1 ,
\]
where the sum goes over submodules of $M$ (cf. \cite[Example~3.2]{CalderoChapoton}). There is also an explicit formula in terms of the shape of the string module $M$ over $B_T$, see \cite[Theorem 4.6]{BFGST18}.

\begin{Rmk}
Here we only scratch the surface of cluster characters and cluster structures on triangulated/Frobenius/extriangulated categories, for a more in depth treatment see \cite{BIRS}.
\end{Rmk}

We end this section with an overview over the correspondences between the combinatorial model (triangulated polygon) $P_n$, Grassmannian cluster algebras $\mc{A}(2,n)$ and Grassmannian cluster categories $\mc{C}(2,n)$ and CC-friezes and their various properties in Figure \ref{Fig:Correspondence}. It is left as an exercise for the reader to adapt the correspondences to the classical case of cluster algebras of type $A_{n}$ and (generalized) cluster categories of type $A_n$.

\subsection{Pl\"ucker friezes and tame integral $SL_k$-friezes} \label{Sec:Pluecker}

We start with defining Pl\"ucker friezes that are obtained from the Grassmannian cluster algebras $\mc{A}(k,n)$. These were first studied in \cite{BFGST2}, which is the main reference for this section. Again, with the use of the specialized cluster character on $\mc{C}(k,n)$ one can obtain tame integral $SL_k$-friezes. However, for $k \geq 3$ not all such friezes are obtained this way. We will in particular discuss the missing finite type cases and pose some questions. \\
Note that recently Pl\"ucker friezes were used to study properties of $SL_k$-tilings, see \cite{PetersenSerhiyenko}.

\begin{landscape}
\begin{figure}[h]
{\tiny
\begin{tabular}{|c|c|c|c|}
\hline
{\bf Combinatorics} & {\bf CC-Frieze } & \multicolumn{2}{|c|}{\begin{minipage}{10cm}
\begin{minipage}{4cm} \begin{centering}{\bf Cluster Algebra} \end{centering} \end{minipage} \begin{minipage}{2cm} $\xleftarrow{\text{Cluster Character } \cc}$\end{minipage} \begin{minipage}{4cm}  {\bf \phantom{lllll}Cluster Category} \end{minipage}
\end{minipage} } \\
\hhline{|=|=|=|=|}
\begin{minipage}{4cm} Polygon $P_n$ with all diagonals \\
and all boundary edges \\
\begin{tikzpicture}[scale=1.5]

\foreach \i in {1,...,6}{
  \coordinate (A\i) at (60*\i:1);
}

\draw[thick] (A1)--(A2)--(A3)--(A4)--(A5)--(A6)--cycle;

\draw[thick] (A1)--(A3);
\draw[thick] (A1)--(A4);
\draw[thick] (A1)--(A5);
\draw[thick] (A2)--(A4);
\draw[thick] (A2)--(A5);
\draw[thick] (A2)--(A6);
\draw[thick] (A3)--(A5);
\draw[thick] (A3)--(A6);
\draw[thick] (A4)--(A6);

    \fill (A1) circle (1pt) node[shift={(60:0.25)}] {$j$};
   \fill (A4) circle (1pt) node[shift={(180:0.3)}] {$i$};

\end{tikzpicture}
\end{minipage} &  
\begin{minipage}{5cm} 
CC-frieze of width $w=n-3$ \\

\begin{tikzpicture}[scale=0.4]
\draw (0,6)  node (M1n)   {$1$};
\draw (2,6)  node (Mn-1n) {$1$};
\draw (4,6)  node (Mn-2n-1) {$1$};
\draw (10,6) node (M23)   {$1$};
\draw (12,6) node (M12r)  {$1$};
\draw (1,5)  node (M1n-1) {$m_{1,3}$};
\draw (3,5)  node (Mn-2n) {$m_{2,4}$};
\draw (9,5)  node (M24)   {$m_{n-2,n}$};
\draw (11,5) node (M13r)  {$m_{n-1,1}$};
\draw (2,4)  node (M1n-2) {$m_{1,4}$};
\draw (8,4)  node (M25)   {$m_{n-3,n}$};
\draw (10,4) node (M14r)  {$m_{n-2,1}$};
\draw (4,2)  node (M13)   {$m_{1,n-1}$};
\draw (6,2)  node (M2n)   {$m_{2,n}$};
\draw (8,2)  node (M1n-1r){$m_{1,3}$};
\draw (5,1)  node (M12)   {$1$};
\draw (7,1)  node (M1nr)  {$1$};

%

\foreach \x/\y/\a/\b in {6/6/n-3/n-2, 8/6/3/4, 5/5/n-3/n-1, 7/5/3/5,
                         4/4/n-3/n, 3/3/1/n-3, 5/3/n-4/n, 7/3/2/6, 9/3/1/5}
{\draw (\x,\y) node (M\a\b) {};}

\foreach \s/\t in {n-2n-1/n-3n-1, n-2n/n-3n, 1n-2/1n-3, 34/24, 35/25, 26/25, 15/14r}
{\path [dotted] (M\s) edge (M\t);}

\foreach \s/\t in {1n-3/13, n-3n/n-4n, 2n/26, 1n-1r/15, n-2n-1/34, n-3n-1/35}
{\path [dotted] (M\s) edge (M\t);}
\end{tikzpicture}
\end{minipage}
 & 
 \begin{minipage}{5cm}
 Grassmannian cluster algebra \\
 \[\mc{A}(2,n)=\CC[p_{ij}: 1 \leq i < j \leq n]/ I_P \ ,\]
 where $I_P$ is ideal generated by Pl\"ucker relations
 \end{minipage}
 & 
 \begin{minipage}{5cm} 
 Grassmannian cluster category \\
 \[ \mc{C}(2,n) \]
 \end{minipage}
 \\
\hline
\begin{minipage}{4cm}arcs $(i,j)$ in $P_n$  \\
(with $1 \leq i < j \leq n$ considered cyclically modulo $n$) \\
\begin{tikzpicture}[scale=1.5]

\foreach \i in {1,...,6}{
  \coordinate (A\i) at (60*\i:1);
}

\draw[thick] (A1)--(A2)--(A3)--(A4)--(A5)--(A6)--cycle;

\draw[thick] (A1)--(A4);

    \fill (A1) circle (1pt) node[shift={(60:0.25)}] {$j$};
   \fill (A4) circle (1pt) node[shift={(180:0.3)}] {$i$};

\end{tikzpicture}
\end{minipage}
& 
\begin{minipage}{5cm}
entries $m_{ij}$ in fundamental domain \\
(with $1\leq i < j \leq n$ and $|i-j|>1$, and $m_{i,i+1}=1$, considered cyclically modulo $n$) 
\end{minipage} 
&  
cluster variables $p_{ij}$ in $\mc{A}(2,n)$ 
 &  
 rigid rank $1$ modules $M_{ij}$ in $\mc{C}(2,n)$ \\
\hline
\begin{minipage}{4cm} Triangulation {\color{red} $T$} of $P_n$\\
\begin{tikzpicture}[scale=1.5]

\foreach \i in {1,...,6}{
  \coordinate (A\i) at (60*\i:1);
}

\draw[thick,color=red] (A1)--(A2)--(A3)--(A4)--(A5)--(A6)--cycle;

\draw[thick, color=red] (A1)--(A3);
\draw[thick, color=red] (A1)--(A4);
\draw[thick, color=red] (A1)--(A5);

    \fill (A1) circle (1pt) node[shift={(60:0.25)}] {$6$};
        \fill (A2) circle (1pt) node[shift={(60:0.25)}] {$5$};
     \fill (A5) circle (1pt) node[shift={(120:-0.25)}] {$2$};
   \fill (A4) circle (1pt) node[shift={(180:0.3)}] {$3$};
   \fill (A3) circle (1pt) node[shift={(240:0.3)}] {$4$};
      \fill (A6) circle (1pt) node[shift={(360:0.3)}] {$1$};

\end{tikzpicture}
\end{minipage}  
& 
\begin{minipage}{5cm}
maximal collections of $1$'s in a fundamental region \\
{\tiny
 \xymatrix@=0.3em{ 
   &\ldots & \rot{1} && \rot{1} && \rot{1} && \rot{1} && \rot{1} &   \ldots &\\
  &\ldots && \rot{1} && 2 && 2 && 2 &&    \ldots& \\
  &\ldots & \grau{3} && \rot{1} && 3 && 3 &&  \grau{1}  &  \ldots& \\
 &\ldots && \grau{2} && \rot{1} && 4 && \grau{1} &&    \ldots& \\
   &\ldots & \grau{1} && \grau{1} && \rot{1} && \grau{1} && \grau{1} &    \ldots & \\
  }
}
\end{minipage}
&
\begin{minipage}{4cm} $\{ p_{ij} \}$ maximal collection \\ of non-crossing Pl\"ucker coordinates
\end{minipage} & 
\begin{minipage}{4cm} 
cluster tilting object 
\[ T=\bigoplus_{l=1}^{n-3} M_{i_l,j_l} \oplus \bigoplus_{l=1}^n M_{l,l+1} \ , \]
where $|i_l-j_l|>1$ and $T$ is maximally rigid  in $\mc{C}(2,n)$ 
\end{minipage}\\
\hline
\begin{minipage}{5cm} flip of a diagonal $(i,j)$ in $T$ \\
\begin{minipage}{2cm}\centering
\begin{tikzpicture}[scale=0.8]

\foreach \i in {1,...,6}{
  \coordinate (A\i) at (60*\i:1);
}

\fill[lightgray] (A1)--(A3)--(A4)--(A5)--cycle;
\draw[thick] (A1)--(A2)--(A3)--(A4)--(A5)--(A6)--cycle;

\draw[thick] (A1)--(A3);
\draw[thick,color=red] (A1)--(A4);
\draw[thick] (A1)--(A5);

    \fill (A1) circle (1pt) node[shift={(60:0.25)}] {$i$};
     \fill (A5) circle (1pt) node[shift={(120:-0.25)}] {$k$};
   \fill (A4) circle (1pt) node[shift={(180:0.3)}] {$j$};
   \fill (A3) circle (1pt) node[shift={(240:0.3)}] {$l$};

\end{tikzpicture}
\end{minipage}
\begin{minipage}{0.5cm}
$\xrightarrow{\mu_{ij}}$
\end{minipage}
\begin{minipage}{2cm} \centering
\begin{tikzpicture}[scale=0.8]

\foreach \i in {1,...,6}{
  \coordinate (A\i) at (60*\i:1);
}

\fill[lightgray] (A1)--(A3)--(A4)--(A5)--cycle;
\draw[thick] (A1)--(A2)--(A3)--(A4)--(A5)--(A6)--cycle;

\draw[thick] (A1)--(A3);
\draw[thick] (A1)--(A5);
\draw[thick,color=red] (A3)--(A5);

    \fill (A1) circle (1pt) node[shift={(60:0.25)}] {$i$};
     \fill (A5) circle (1pt) node[shift={(120:-0.25)}] {$k$};
   \fill (A4) circle (1pt) node[shift={(180:0.3)}] {$j$};
   \fill (A3) circle (1pt) node[shift={(240:0.3)}] {$l$};

\end{tikzpicture}
\end{minipage}

\end{minipage} & \begin{minipage}{4cm}frieze mutation at entry $m_{ij}$, \\see \cite{BFGST18} for an extensive treatment \end{minipage}& cluster mutation at $p_{ij}$ & $\add(T/ M_{ij})$-approximation  \\
\hline
\end{tabular}
}
\caption{Correspondences} \label{Fig:Correspondence}
\end{figure}
\end{landscape}

We need to recall some terminology about Pl\"ucker coordinates in $\mc{A}(k,n)$ from \cite{BFGST2}. Throughout we assume $k \leq \frac{n}{2}$ and we consider sets $I$ modulo $n$, where we choose the representatives 
$1, \ldots, n$ for the equivalence classes modulo $n$. The Pl\"ucker coordinates $p_I$  can be indexed by $k$-subsets of $[n]=\{1, \ldots, n\}$. If $I$ consists of $k$ consecutive elements, modulo $n$, then
$I$ is called a \emph{consecutive $k$-subset} and 
$p_I$ a {\em consecutive Pl\"ucker coordinate}. These are frozen cluster variables or 
coefficients. Further, if $I$ is of the form $I=I_0\cup \{m\}$ where $I_0$ is a consecutive $(k-1)$-subset of $[n]$ and $m$ is any entry of $[n]\setminus I_0$, then $I$ is {\em almost consecutive}, and $p_I$ is an \emph{almost consecutive Pl\"ucker coordinate}.
For a tuple $J=\{i_1,\dots, i_{\ell}\}$ 
of $[n]$ and $\{a_1, \ldots, a_\ell\} = \{i_1, \ldots, i_\ell\}$ with $1\le a_1\le \dots\le a_{\ell}\le n$ 
we set 
\[
o(J)=o(i_1,\dots, i_{\ell}) = (a_1,\dots, a_{\ell}).
\]
With this notation, if $0\le s,\ell\le k$, then $p_{b_1,\dots, b_s,o(i_1,\dots, i_{\ell}),  
c_1,\dots, c_{k-\ell+a}}$ means \newline $p_{b_1,\dots, b_s,a_1,\dots, a_{\ell}, 
c_1,\dots, c_{k-\ell+a}}$. Further use the shorthand notation $[r]^\ell:=\{ r, r+1, \ldots, r+\ell-1\}$.

\begin{defi}\label{Def:Pluecker-frieze}
	The {\em Pl\"ucker frieze} of type $(k,n)$ is a $\ZZ \times \{1, 2, \ldots, n+k-1\}$-grid with entries given by the map
	\begin{eqnarray*}
		\varphi_{(k,n)} \colon \ZZ \times \{1, 2, \ldots, n+k-1\} & \to & \mathcal{A}(k,n)\\
		(r,m) & \mapsto & p_{o([r']^{k-1},m')},
	\end{eqnarray*}
	where $r'\in[n]$ is the reduction of $r$ modulo $n$	and $m'\in[n]$ is the 
	reduction of $m+r'-1$ modulo $n$.  
We denote the Pl\"ucker frieze by $\mathcal P_{(k,n)}$. 
	The {\em special Pl\"ucker frieze of type $(k,n)$}, denoted by $s\mathcal P_{(k,n)}$, 
	is the grid we get from $\mathcal P_{(k,n)}$ 
	by substituting the consecutive Pl\"ucker coordinates with $1$s. 
	\end{defi}
An illustration of the Pl\"ucker frieze $\mc{P}_{(k,n)}$ can be found in \cite[Figure 1]{BFGST2} and $s\mathcal{P}_{(3,6)}$ is in Example \ref{Ex:D4} below. The main Theorem of \cite[Theorem 3.3]{BFGST2} shows that the special Pl\"ucker frieze is a tame $SL_k$-frieze over 
\[ s\mc{A}(k,n)=\mc{A}(k,n)/\langle p_I-1: \text{ $p_I$ is a consecutive Pl\"ucker coordinate} \rangle \ .\]
Further let ${\bf x}=(x_1, \ldots, x_{m})$ be a cluster in $\mc{A}(k,n)$ and consider the map $\varepsilon_{{\bf x}}: \mc{A}(k,n) \xrightarrow{} \CC$ sending each element of ${\bf x}$ to $1$. The map $\varepsilon_{{\bf x}}$ is called \emph{specialization to $1$}. This yields tame integral friezes: 

\begin{Thm}[Corollary 3.11 of \cite{BFGST2}] \label{Thm:SLk-friezes-pluecker}
Let ${\bf x}=(x_1, \ldots, x_{m})$ be a cluster in $\mc{A}(k,n)$ and consider the specialization map $\varepsilon_{{\bf x}}$. Then $\varepsilon_{{\bf x}}(s\mc{P}_{(k,n)})$ is a tame integral $SL_k$-frieze of width $w=n-k-1$.
\end{Thm}

Friezes obtained by specializing a cluster to $1$ are called \emph{unitary friezes}.
For $\mc{A}(2,n)$ we have seen that this construction yields all (finitely many) tame integral $SL_2$-friezes, a.k.a.~CC-friezes, and hence all tame integral $SL_2$-friezes are unitary. \\
 However, for $k \geq 3$ the situation is much less tractable, even for the remaining finite types ($k=3$, $n \in \{ 6,7,8\}$) there are more $SL_3$-friezes than the ones obtained by Theorem \ref{Thm:SLk-friezes-pluecker}. These friezes are called \emph{non-unitary friezes} in the literature, and much less is known about them.
 
 \begin{Ex} 
 \label{Ex:D4}Consider the Grassmannian cluster algebra $\mc{A}(3,6)$. Then the specialized Pl\"ucker frieze is  $s\mathcal P_{(3,6)}$: 
\[
\xymatrix@C=-8pt@R=-1pt{
 && 0 && 0 && 0 && 0 && 0 && 0 && 0 \\
&0 && 0 && 0 && 0 && 0 && 0 && 0 && 0\\
  && 1 && 1 && 1 && 1 && 1 && 1 && 1 \\
&p_{146} && p_{125} && p_{236} && p_{134} && p_{245} && p_{356} && p_{146} && \cdots  \\
\cdots && p_{124} && p_{235} && p_{346} && p_{145} && p_{256} && p_{136} && p_{124} \\
& 1  && 1 && 1 && 1 && 1 && 1 && 1 && 1 \\
 && 0 && 0 && 0 && 0 && 0 && 0 && 0 \\
&0 && 0 && 0 && 0 && 0 && 0 && 0 && 0
}
\]
In $\mc{A}(3,6)$ a cluster is of size $4+6$ and is e.g. given by $p_{245}, p_{145}, p_{124}, p_{125}$ plus the six consecutive Pl\"ucker coordinates. Specializing the non-consecutive Pl\"ucker coordinates in this cluster also to $1$ yields the unitary integral $SL_3$-frieze in Fig.~\ref{Fig:SL3}. However, one can also specialize them to $2$ and get the non-unitary integral $SL_3$-frieze in Fig.~\ref{Fig:SL3}.
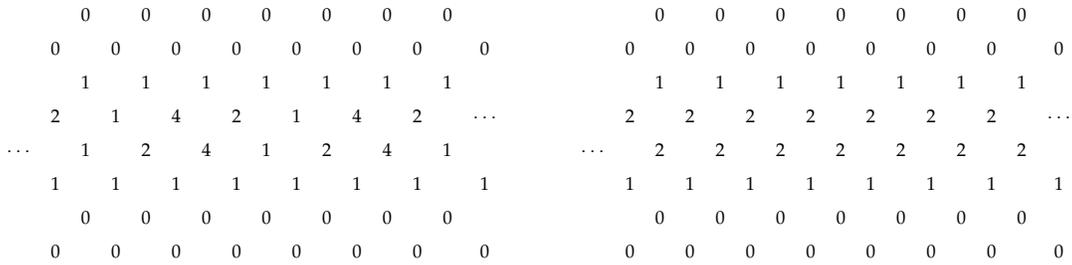
\begin{figure}[h!]
{\tiny
\begin{minipage}{0.4 \textwidth}
\xymatrix@=0.2em{
 && 0 && 0 && 0 && 0 && 0 && 0 && 0 \\
&0 && 0 && 0 && 0 && 0 && 0 && 0 && 0\\
  && 1 && 1 && 1 && 1 && 1 && 1 && 1 \\
&2 && 1 && 4 && 2 && 1 && 4 && 2 && \cdots  \\
\cdots && 1 && 2 && 4 &&  1 && 2 && 4 && 1 \\
& 1  && 1 && 1 && 1 && 1 && 1 && 1 && 1 \\
 && 0 && 0 && 0 && 0 && 0 && 0 && 0 \\
&0 && 0 && 0 && 0 && 0 && 0 && 0 && 0
}
\end{minipage}
\begin{minipage}{0.051 \textwidth}
$\phantom{a}$
\end{minipage}
\begin{minipage}{0.4 \textwidth}
\xymatrix@=0.2em{
 && 0 && 0 && 0 && 0 && 0 && 0 && 0 \\
&0 && 0 && 0 && 0 && 0 && 0 && 0 && 0\\
  && 1 && 1 && 1 && 1 && 1 && 1 && 1 \\
&2 && 2 && 2 && 2 &&2 && 2&& 2 && \cdots  \\
\cdots && 2 && 2 && 2 && 2 && 2 && 2 && 2 \\
& 1  && 1 && 1 && 1 && 1 && 1 && 1 && 1 \\
 && 0 && 0 && 0 && 0 && 0 && 0 && 0 \\
&0 && 0 && 0 && 0 && 0 && 0 && 0 && 0
}
\end{minipage}
}
\caption{A unitary $SL_3$-frieze (left) and a non-unitary $SL_3$-frieze (right).} \label{Fig:SL3}
\end{figure}
Note that the AR-quiver of the stable category $\underline{\mc{C}(3,6)}$ is of type $D_4$.
 \end{Ex}
 
We give an overview over some results about the finite type friezes below, in particular about their enumeration. One can again use  Grassmannian cluster categories $\mc{C}(k,n)$ and the relation to so-called mesh friezes, which are obtained from the AR-quiver of $\mc{C}(k,n)$:

\begin{defi}[cf. Def.~4.6 of \cite{BFGST2}]
Let $\mc{C}(k,n)$ be of finite type. 
A {\em mesh frieze} $\mathcal{F}_{k,n}$ for  $\mc{C}(k,n)$ is a collection of positive integers, one for each indecomposable object in $\mc{C}(k,n)$ such that $\mathcal{F}_{k,n}(P)=1$ for every indecomposable projective-injective $P$ and such 
that all mesh relations on the AR-quiver of $\mc{C}(k,n)$ evaluate to 1. This means, that
for any AR-sequence 
\begin{equation}\label{ar-triangle}
0\to A\to \bigoplus_{i=1}^\ell B_i \to C \to 0,
\end{equation}
where the $B_i$ are indecomposable, and $\ell \geq 1$, we have
\[
	\mathcal{F}_{k,n}(A)\mathcal{F}_{k,n}(C) = \prod_{i=1}^\ell \mathcal{F}_{k,n}(B_i) + 1.
\]
\end{defi}

Note that in \cite{ARS} certain friezes were defined for Cartan matrices associated to acyclic quivers, also cf.~\cite{KellerScherotzke}, where these friezes were interpreted in a categorical framework. For Dynkin type these friezes correspond to the stable versions $\underline{\mc{F}_{k,n}}$ of mesh-friezes defined above: $\underline{\mc{F}_{2,n}}$ yields an $A_{n-3}$ Dynkin type frieze, and $\underline{\mc{F}_{3,n}}$ for $n=6,7,8$ yields type $D_4$, $E_6$, $E_8$, respectively.

It is shown in  \cite[Cor.~4.16]{BFGST2} that there is a bijection between
$$\{ \text{mesh friezes for $\mathcal{C}(3,n)$} \} \longleftrightarrow \{ \text{ tame integral }SL_3-\text{friezes of width } n-4 \} \ ,$$
for $n \in \{6,7,8\}$. 

\begin{Rmk}
The numbers of the finite type mesh friezes are as follows:\\
\begin{center}
\begin{tabular}{|c|c|c|c|c|}
\hline
Type & $\mc{C}(2,n)$ & $\mc{C}(3,6)$ & $\mc{C}(3,7)$  & $\mc{C}(3,8)$ \\
\hhline{|=|=|=|=|=|}
unitary & $C_{n-2}$ & 50 & 833 &  25080 \\
\hline
non-unitary & 0 & 1 & 35 &  1872 \\
\hline
Total & $C_{n-2}$ & 51 &  868 & 26952 \\
\hline
\end{tabular}
\end{center}
For type $\mc{C}(2,n)$ mesh friezes are just the CC-friezes, which are all unitary, and their number has been confirmed by Conway and Coxeter. The type $\mc{C}(3,6)$ friezes have been studied as $2$-friezes by Morier-Genoud, Ovsienko, and Tabachnikov in \cite{MGOT12}, where also the number $51$ was proven. The numbers for $\mc{C}(3,7)$ and $\mc{C}(3,8)$ were conjectured by Fontaine and Plamondon in \cite{FP16} (for $\mc{C}(3,7)$ also see \cite{Propp,MGOT12}). For $\mc{C}(3,7)$ the conjecture was confirmed by Cuntz and Plamondon in the Appendix of \cite{BFGST2}. Both proofs use arithmetic properties of the entries of the friezes. Finally, a  preprint by Zhang \cite{Zhang-finite} confirms the number for $\mc{C}(3,8)$ with diophantine geometry methods, using results by Gunawan and Muller  about the finiteness of these friezes \cite{GunawanMuller, Muller-finiteD} and frieze points \cite{dSGHL}.
\end{Rmk}

\begin{Rmk}
In the cases  $\mc{C}(3,n)$, where $n \in \{6,7,8\}$, one can explain the occurrence of all non-unitary friezes from the non-unitary frieze of Example \ref{Ex:D4} by applying Iyama--Yoshino reduction on $\mathcal{C}(3,n)$, see \cite[Section 5]{BFGST2}. The key idea here is to perform Iyama--Yoshino reduction (see \cite{IY08}, also \cite{BIRS}) at a suitable indecomposable module on the Frobenius category $\mathcal{C}(3,n)$ and observing that one again obtains a mesh frieze of smaller Dynkin type this way, see \cite[Prop 5.3]{BFGST2} for an ad-hoc proof and \cite[Section 8]{FMP} for a more conceptual approach.
\end{Rmk}

\begin{Rmk}
Iyama--Yoshino reduction has been studied for various categories, in particular for $\mc{C}(k,n)$, see \cite{FMP}. It is a purely categorical construction, but for $\mc{C}(2,n)$ it has a nice analogue in the combinatorial model: Let $T$ be a triangulation of $P_n$, and let $(i,j) \in T$ be a diagonal. Then IY-reduction with respect to the corresponding module $M_{ij} \in \mc{C}(2,n)$ yields two smaller CC-friezes that are ``glued'' together at the entry $M_{ij}$, see \cite[Example 8.8]{FMP}. These are precisely obtained by cutting $P_n$ along $(i,j)$ and computing the friezes for the two resulting smaller polygons. However, if one reduced with respect to a diagonal $(i,j) \not \in T$, then one obtains two smaller friezes with coefficients. \\
It would also be interesting to study the reverse operation.
\end{Rmk}

In particular, for infinite type cases, non-unitary friezes are not well understood yet. Therefore we post the following

\begin{Qu}
\begin{enumerate}[(1)]
\item Can we find a representation theoretic explanation of non-unitary $SL_k$-friezes from Grassmannians $\mc{C}(k,n)$ for $k \geq  3$ and $n \geq  9$? In particular, are there always tame integral $SL_k$-friezes whose nontrivial entries are all in $\ZZ_{>1}$?
\item More from a quiver perspective: consider an initial cluster of Pl\"ucker coordinates  in $\mc{A}(k,n)$. When is it possible  to specialize all its values to integers $>1$ such that all the other Pl\"ucker coordinates are also positive integers?\footnote{An example for this is the cluster for $\mc{A}(3,6)$ of Example \ref{Ex:D4}, which can be specialized to $2$.} And can one find representation theoretic interpretation for this phenomenon?
\end{enumerate}
\end{Qu}

\subsection{Infinite frieze structures} \label{Sub:infinite}

In this section we consider a generalization for the correspondences in the table in Fig.~\ref{Fig:Correspondence} starting from the  natural question: what happens when $n \xrightarrow{} \infty$, that is, when one considers triangulations of an $\infty$-gon? 

We  sketch how construct a suitable Frobenius categorification $\mc{C}(2,\infty)$ of the (completed) $\infty$-gon from $\mc{C}(2,n)$ below, following \cite{ACFGS}. Then we obtain frieze structures from an adapted cluster character, which will appear in detail in \cite{EsentepeFaber}.

Consider the \emph{$\infty$-gon}, that is, a disc with a discrete set of marked points on the boundary admitting a unique two-sided accumulation point. The \emph {completed $\infty$-gon} is obtained by adding the unique accumulation point as a marked point.  Usually, the marked points are labeled by $\ZZ$, and the accumulation point will be denoted by $\infty$. Similar as for the finite polygon case, we call $(a,b)$ an {\em arc of the completed $\infty$-gon} in $\ZZ \cup \{\infty\}$, such that $a < b$. There are two types of arcs: $(a,b)$ is {\em finite} if $a,b \in \ZZ$, and {\em infinite} if $b = \infty$. \\
One defines a \emph{triangulation of the (completed) $\infty$-gon} as a maximal set of non-crossing arcs (where crossing is defined as in the finite case). Triangulations and mutations have been studied independently by \cite{BaurGratz} and \cite{CanakciFelikson}. There are five types of triangulations of the completed $\infty$-gon, \cite[Thm.~1.12]{BaurGratz} - we will only be discussing so-called \emph{fan triangulations} here: a fan triangulation $T$ of the completed $\infty$-gon has a unique point $a \in \ZZ$ such that there are infinitely many arcs of the form $(a,b) \in T$ with $b \in \ZZ$ and $b>a$ as well as infinitely many arcs of the form $(c,a) \in T$ with $c \in \ZZ$ and $c <a$, and $(a,\infty)\in T$.

\begin{Rmk}
We want to point out that $\mc{C}(2, \infty)$ is one possible construction of a category with infinity type $A$ cluster structure. It is building on the foundational paper \cite{HJ-infinite}, where a triangulated version has been defined for the non-completed $\infty$-gon. Moreover, there are several other approaches: \cite{PY21} construct a triangulated category for the $\infty$-gon allowing for several accumulation points via Igusa--Todorov's \cite{ITcyclic} approach, also see \cite{MohammadiRockZaffalon} for more work in this direction. Fisher \cite{F16} constructs  a completion of the Holm--J{\o}rgensen category. These approaches are compared in \cite{ACFGS}. Moreover,  recently Canakc\i, Kalck, and Pressland  \cite{CKP} constructed an extriangulated category and show that the weak cluster tilting subcategories of this extriangulated category are indeed in bijection with all the triangulations of the completed $\infty$-gon.
\end{Rmk}

\begin{Rmk}
We will not be focussing on the cluster \emph{algebra} structure here, but there is also ongoing work in this direction:
Let 
$$\mc{A}(k,\infty)=\CC[p_I: I \subseteq \ZZ, |I|=k] / I_P $$
where $I_P$ is the ideal of Pl\"ucker relations, which are defined as follows: for any two subsets $J,J' \subseteq \ZZ$ with $|J| = k+1$ and $|J'| = k-1$, with $J = \{j_0, \ldots, j_{k}\}$ and $j_0 < \ldots < j_{k}$ we get a Pl\"ucker relation
\begin{eqnarray}
	\sum_{l=0}^{k}(-1)^lp_{J' \cup \{j_l\}}p_{J \setminus \{j_l\}} \ .
\end{eqnarray}
Then $I_P$ is the ideal generated by these (infinitely many) Pl\"ucker relations. Here the $p_I$ are also called Pl\"ucker coordinates and we also speak of non-crossing $p_I$ and $p_J$ if the sets $I$ and $J$ satisfy the same conditions as in the finite case.
In \cite{GratzGrabowski}, and its Appendix by Groechenig, it has been shown that $\mc{A}(k,\infty)$ can be equipped with a cluster structure and indeed this ring can be interpreted as the homogeneous coordinate ring of an infinite version of a Grassmannian under a generalization of the Pl\"ucker embedding. Also see a recent discussion of these rings as colimits in \cite{GratzKorff}.
\end{Rmk}

Fixing $k \geq 2$, we want to obtain a category $\mc{C}(k,\infty)$ from the JKS category 
\[\mc{C}(k,n)=\CM_{C_n}(\CC\llbracket x,y \rrbracket/\langle x^k - y^{n-k} \rangle) \ .\] Consider the action of the multiplicative group $\mathbb{G}_m$ on $\CC[x,y]$, with 
$$x \mapsto \zeta x \text{ and } y \mapsto \zeta^{-1}y \ .$$
Take as semi-invariant $x^k$ (since $y^{n-k}$ tends to be ``small'' for $n \gg 0$) and consider the category
\[ \CM_{\mathbb{G}_m}(\CC[x,y]/\langle x^k \rangle)  \]
the category of $\mathbb{G}_m$-equivariant $\CM$-modules over $R_k:=\CC[x,y]/\langle x^k \rangle$. One can show that this is equivalent to the category
\[ \mc{C}(k, \infty):=\CM_{\ZZ}(R_k) \ , \]
where the $\ZZ$-grading is given by $\deg x=1$ and $\deg y=-1$. Then $\mc{C}(k, \infty)$ is called the \emph{Grassmannian cluster category of infinite rank (of type $k$)}. Note that in this context, all modules, ideals, etc. are considered w.r.t. the grading.
In particular, we want to remark here, that although we have to consider graded modules, the rings $R_k$ are commutative hypersurface rings (in particular they are Gorenstein), which allows us to use tools from commutative algebra: in particular, for hypersurface rings one can use Eisenbud's matrix factorizations \cite{EisenbudMF} to explicitly write down $\CM$-modules.\footnote{However, under the $\ZZ$-grading, the degree $0$ part of $R_k$ is not a field, and also the grading is not bounded -- which makes it hard to handle for computer algebra systems.}

\begin{Thm}[see Thm. 3.7, 3.9, and 4.1 of \cite{ACFGS-IMRN} for precise statements] 
The generically free\footnote{This condition is sometimes also called ``free on the punctured spectrum''.} rank $1$ modules over $R_k$ are explicitly identified with monomial ideals in $R_k$, which yields a $1-1$ correspondence with the Pl\"ucker coordinates $p_I$. Moreover, two Pl\"ucker coordinates $p_I$ and $p_J$ in $R_k$ are compatible if and only if for the corresponding modules $M_I, M_J$ in $\mc{C}(k,\infty)$ satisfy $\Ext^1_{\mc{C}(k, \infty)}(M_I, M_J)=0$.
\end{Thm}

Below we will see that for $k=2$, there is a complete classification of the indecomposable objects in $\mc{C}(2,\infty)$. However for $k \geq 3$ the rings $R_k$ are of wild Cohen--Macaulay type and not much is known about their module theory.

\begin{Qu}
Let  $k \geq 3$. Can we describe generically free modules of higher rank in $\mc{C}(k,\infty)$ that correspond to cluster variables of higher degree in $\mc{A}(k,\infty)$? Is there a combinatorial way to describe a cluster structure with Postnikov diagrams, similar to the finite case in \cite{Scott06}?
\end{Qu}

Now let us focus on the case $k=2$, that is, the ring $R_2=\CC[x,y]/\langle x^2 \rangle$ with the $\ZZ$-grading as above. This ring (the ungraded version) has appeared as $A_\infty$-singularity in singularity theory \cite{Siersma} and as one of the two examples of hypersurface rings of countable $\CM$-type \cite{BGS} (the other one being the $D_\infty$-singularity $\CC[x,y]/ \langle x^2y \rangle$). Moreover, Buchweitz, Greuel, and Schreyer have classified all indecomposable $\CM$-modules. Combining this with the results on Pl\"ucker coordinates and types of arcs in the $\infty$-gon, one obtains a one-to-one correspondence between arcs in the $\infty$-gon and $\CM$-modules in $\mc{C}(2,\infty)$.

One defines cluster tilting subcategories of the Frobenius category $\mc{C}(2,n)$ as functorially finite maximally rigid subcategories. It can be shown, see \cite{PY21, ACFGS}, that these correspond precisely to the fountain triangulations of the $\infty$-gon. Further, the other types of triangulations can be shown to be maximal almost rigid subcategories of $\mc{C}(2,\infty)$ \cite[Theorem 5.5]{ACFGS}. Moreover, mutation can be defined in the $\infty$-gon, cf.~\cite{BaurGratz}, and one can show that it corresponds to a categorical mutation, cf.~\cite{IY08}, in $\mc{C}(2,\infty)$, see \cite[Theorem 5.9]{ACFGS}. We get the following analogue of Table \ref{Fig:Correspondence}:

\begin{Thm}
There are one-to-one correspondences between arcs in a completed $\infty$-gon and indecomposable objects in the Frobenius category $\mc{C}(2,\infty)$ as indicated in the table in Fig.~\ref{Fig:Correspondence-infinite}.
\begin{figure}[h!]
\begin{tabular}{|c|c|}
\hline
{\bf Combinatorial model }& {\bf Frobenius category} \\
\hhline{|=|=|}
arcs in completed $\infty$-gon
&
indecomposable objects of $\mc{C}(2,\infty)$
\\
\hline
\begin{minipage}{0.45 \textwidth}
    \centering
    \begin{tikzpicture}[scale=3,cap=round,>=latex]
    \draw[black](270:0.5) .. controls (270:0.2) and (200:0.2) .. (200:0.5);
    \node at (270:0.65) {$-l+1$};
    \node at (200:0.75) {$-l-j$};
    \draw[black] (0,0) circle(0.5cm);
    \node at (90:0.65) {$\infty$};
	\draw (90:0.5) node[fill=black,circle,inner sep=0.065cm] {} circle (0.02cm);	        
 \end{tikzpicture}

finite arcs: $(-l-j, -l+1)$ \\
(including the boundary arcs $(-l,-l+1)$
\end{minipage}
 & \begin{minipage}{0.45 \textwidth}
 \centering
 ideals $\langle x,y^j \rangle(l)$ in $R_2$ \\
 (including projective-injectives $R(l)$)
 \end{minipage}\\
\hline
\begin{minipage}{0.45 \textwidth}
   \centering
    \begin{tikzpicture}[scale=3,cap=round,>=latex]
    \draw[black](300:0.5) .. controls (300:0.2) and (90:0.2) .. (90:0.5);
    \node at (300:0.65) {$-l$};
    \draw[black] (0,0) circle(0.5cm);
    \node at (90:0.65) {$\infty$};
	\draw (90:0.5) node[fill=black,circle,inner sep=0.065cm] {} circle (0.01cm);	        
 \end{tikzpicture} \\
infinite arcs $(-l, \infty)$
\end{minipage} & the modules $\CC[y](l)$ \\
\hline
\begin{minipage}{0.45 \textwidth}
\centering
 
 \begin{tikzpicture}[scale=3,cap=round,>=latex]

  \foreach \a in {275, 280, 285, 290, 295,300, 305, 310, 315, 320, 325, 330, 335, 340,345, 350, 355, 360,365,0,5,10,15,20,25,30,35,40,45,50,55,60} {
    \draw[red] (270:0.5) .. controls (270:0.35) and (\a:0.35) .. (\a:0.5);
  }
  
    \foreach \a in {265,260,255,250,245,240,235,230,225,220,215,210,205,200,195,190,185,180,175,170,165,160,155,150,145,140,135,130,125} {
    \draw[red] (270:0.5) .. controls (270:0.35) and (\a:0.35) .. (\a:0.5);
  }
  
\draw[red](270:0.5) .. controls (270:0.2) and (90:0.2) .. (90:0.5);

  \draw[thick,dotted,color=red] (65:0.4) arc (65:80:0.4);
    \draw[thick,dotted, color=red] (120:0.4) arc (120:95:0.4);

  \draw[red] (0,0) circle(0.5cm);

  \node at (90:0.65) {$\infty$};
  \draw (90:0.5) node[fill=black,circle,inner sep=0.065cm] {} circle (0.015cm);	  

\end{tikzpicture}

fountain triangulations  ${\color{red} T}$
\end{minipage} & cluster tilting subcategories $\mc{T}$ of $\mc{C}(2,\infty)$ \\
\hline
\begin{minipage}{0.45 \textwidth}
admissible mutation as defined in \cite[Def.~2.1]{BaurGratz}  
\end{minipage} & categorical mutation as defined in \cite{ACFGS} \\
\hline
\end{tabular}
\caption{The infinite correspondences.} \label{Fig:Correspondence-infinite}
\end{figure}
\end{Thm}

Finally, since the topic of this article is friezes, we want to construct infinite analogues of CC-friezes, that is, some kind of nonperiodic integral $SL_2$-friezes. The idea is to also use a cluster character on $\mc{C}(2,\infty)$, specialize to $1$, and obtain the frieze relations from the properties of this cluster character, similar to Prop.~\ref{Prop:frieze-from-CC}. In \cite{EsentepeFaber} we extend the cluster character of Paquette and Yildirim \cite{PY21} to this case and show how to get frieze structures from fountain triangulations of the completed $\infty$-gon. However, the existence of the fountain point that has infinitely many arcs of the triangulation emanating complicates the generalization and we only obtain so-called \emph{half-friezes}. 

\begin{Ex}
Consider the following fountain triangulation of the $\infty$-gon:
\[ T=\{ (0,2n), n \in \ZZ \backslash \{0, \pm 1 \} \} \cup \{ (2n,2n+2), n \in \ZZ\} \cup \{(n,n+1), n \in \ZZ \} \ . \]
For a schematic picture of the triangulation see Fig.~\ref{Fig:fan-infinity}.
\begin{figure}[h!]
\begin{tikzpicture}[scale=3,cap=round,>=latex]

  \foreach \a in {  290,  310,    330,   350, 10,30,50} {
    \draw[red] (270:0.5) .. controls (270:0.35) and (\a:0.35) .. (\a:0.5);
  }

    \foreach \a in {  300,  320,    340,   360, 20,40} {
    \draw[red] (\a-10:0.5) .. controls (\a-10:0.45) and (\a+10:0.35) .. (\a+10:0.5);
  }
  
  \foreach \a in {270, 280,290,300,310,320,330,340,350,360,10,20,30,40,50,60} {
    \node[fill=black,circle,inner sep=0.02cm] at (\a:0.5) {};
  }
  
    \foreach \a in {250,230,210,190,170,150,130} {
    \draw[red] (270:0.5) .. controls (270:0.35) and (\a:0.35) .. (\a:0.5);
  }
  
    \foreach \a in {240, 220, 200, 180, 160, 140} {
    \draw[red] (\a-10:0.5) .. controls (\a-10:0.45) and (\a+10:0.35) .. (\a+10:0.5);
  }
  
  \foreach \a in {260, 250,230,210,190,170,150,130,240, 220, 200, 180, 160, 140} {
    \node[fill=black,circle,inner sep=0.02cm] at (\a:0.5) {};
  }
  
\draw[red](270:0.5) .. controls (270:0.2) and (90:0.2) .. (90:0.5);

  \draw[thick,dotted,color=red] (65:0.4) arc (65:80:0.4);
    \draw[thick,dotted, color=red] (120:0.4) arc (120:95:0.4);

  \draw[red] (0,0) circle(0.5cm);

  \node at (90:0.65) {$\infty$};
  \draw (90:0.5) node[fill=black,circle,inner sep=0.065cm] {} circle (0.015cm);	  

\end{tikzpicture}
\caption{A fountain triangulation of $\infty$-gon.}  \label{Fig:fan-infinity}
\end{figure}
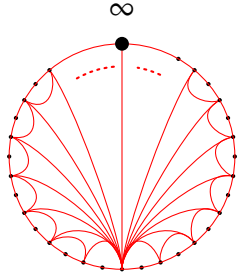

Using the cluster character, the corresponding half-frieze is
\[
{\tiny
\begin{tikzcd}[row sep=0.35em, column sep=0.75em]
\cdots  && 0 &&0 && 0 && 0 && 0&& 0  &&0&&  0&& 0 && 0 && 0 &&  0 && \cdots \\ 
\cdots &1 &&  1 &&1 && 1&& 1 && 1 &&1 &&1 && 1&&1&&1 && 1 &&\cdots   \\  
\cdots & & 1 && 4 && 1 && 3 && 1&& \phantom{m}  &&\rot{1}&&  3&& 1&& 4 && 1 &&  4 && \cdots \\ 
\cdots & 3 &&  3 && 3 && 2 && 2 &&  &&  && \rot{2} &&2 && 3 && 3 && 3 && \cdots   \\ 
\cdots&& 8 && 2 && 5 && 1 &&  &&  &&  && \rot{1} && 5&&2 &&8&& 2 && \cdots   \\ 
\cdots& 5&& 5 && 3 && 2 &&  &&  &&  &&  && \rot{2} &&  3 && 5 && 5 &&  \cdots  \\ 
\cdots&& 3 && 7  && 1 &&  &&  && &&  &&  && \rot{1} && 7 && 3 && 12 &&  \cdots\\ 
\cdots&7 && 4 &&2&&  &&  &&  &&  && &&  && \rot{2} && 4 &&7&&  \cdots   \\ 
\cdots &&  9 && 1 && && && && &&  &&  &&  &&\rot{1} && 9 && 4 && \cdots \\ 
\cdots &5 && 2&& &&  &&  &&  &&  && && &&  && \rot{2} && 5&&\cdots     \\ 
\end{tikzcd}
}
\]

Similarly as in the case of CC-friezes the first nontrivial row is obtained by counting triangles incident to each marked point on the circle (excluding the fountain point). Also note that we can look at the red diagonal in the frieze, if we set all values $\neq 1$ to $0$, we obtain the periodic sequence $1,0,1,0, \ldots$. This corresponds to the so-called index sequence of a Penrose tiling of $\RR^2$ and is studied in more detail in \cite{EsentepeFaber}.

\end{Ex}

We end this section with some related questions:

\begin{Qu}
\begin{enumerate}[(1)]
\item Can one also use $\CM$-modules on the $D_\infty$-singularity to categorify the triangulations of a punctured $\infty$-gon? And does this yield infinite friezes of type $D_\infty$?
\item In a different direction: What about cluster tilting subcategories for higher $\mc{C}(k,\infty)$ for $k \geq 3$?
For the categories $\mc{C}(k,\infty)$ can one obtain  an $\infty$-analogue of Pl\"ucker friezes and hence infinite tame integral $SL_k$-friezes?
\end{enumerate}
\end{Qu}

\subsection{Other frieze types} \label{Sub:other-repthy}
There are several other constructions yielding friezes, also coming from representation theory. For example, friezes of type $D$ from punctured disks were investigated by Baur, Marsh \cite{BM09},  infinite periodic friezes of type $\widetilde A$ from annuli by Baur, Canakci, Jacobsen, Kulkarni, Todorov \cite{BCJKT}, infinite periodic friezes of type $\widetilde D$ from twice punctured disks by Baur, Bittmann, Gunawan, Todorov, Y{\i}ld{\i}r{\i}m \cite{BBGTY}, super friezes of type $A$ by Canakci, Fedele, Garcia Elsener, and Serhiyenko \cite{CFGS}, and many more. \\
Instead of discussing them here (which would be out of scope of the current article) we refer the reader again to Anna Felikson's compendium (cf.~\cite{Felikson-frieze-portal}) of research papers (and more!) on friezes:

 \url{https://www.maths.dur.ac.uk/users/anna.felikson/Projects/frieze/frieze-res.html}. 

\subsection*{Acknowledgements} I want to thank the organizers of the ICRA 2024 Shanghai for the invitation to write a proceedings article, in particular Bernhard Keller for his patience. Further, thanks go to the anonymous referees for a very careful reading and valuable suggestions to improve the article, and to \"Ozg\"ur Esentepe and Zahra Nazemian for comments on a previous version. I also want to thank my wonderful collaborators for many discussions about friezes - in particular Karin Baur and Gordana Todorov, who introduced me to this research topic at the very first WINART workshop in 2016.
%


\begin{thebibliography}{MGOST14}

\bibitem[ACF{\etalchar{+}}23]{ACFGS}
Jenny August, Man~Wai Cheung, Eleonore Faber, Sira Gratz, and Sibylle Schroll.
\newblock {Cluster structures for the $A_{\infty}$ singularity}.
\newblock {\em J. Lond. Math. Soc. (2)}, 107(6):2121--2149, 2023.

\bibitem[ACF{\etalchar{+}}24]{ACFGS-IMRN}
Jenny August, Man-Wai Cheung, Eleonore Faber, Sira Gratz, and Sibylle Schroll.
\newblock Categories for {G}rassmannian cluster algebras of infinite rank.
\newblock {\em Int. Math. Res. Not. IMRN}, (2):1166--1210, 2024.

\bibitem[Ami09]{Amiot}
Claire Amiot.
\newblock Cluster categories for algebras of global dimension 2 and quivers
  with potential.
\newblock {\em Ann. Inst. Fourier (Grenoble)}, 59(6):2525--2590, 2009.

\bibitem[ARS97]{ARS-Book}
Maurice Auslander, Idun Reiten, and Sverre~O. Smal\o.
\newblock {\em Representation theory of {A}rtin algebras}, volume~36 of {\em
  Cambridge Studies in Advanced Mathematics}.
\newblock Cambridge University Press, Cambridge, 1997.
\newblock Corrected reprint of the 1995 original.

\bibitem[ARS10]{ARS}
Ibrahim Assem, Christophe Reutenauer, and David Smith.
\newblock Friezes.
\newblock {\em Adv. Math.}, 225(6):3134--3165, 2010.

\bibitem[ASS06]{ASS}
Ibrahim Assem, Daniel Simson, and Andrzej Skowro\'{n}ski.
\newblock {\em Elements of the representation theory of associative algebras.
  {V}ol. 1}, volume~65 of {\em London Mathematical Society Student Texts}.
\newblock Cambridge University Press, Cambridge, 2006.
\newblock Techniques of representation theory.

\bibitem[Bau21a]{KarinSurvey}
Karin Baur.
\newblock Frieze patterns of integers.
\newblock {\em Math. Intelligencer}, 43(2):47--54, 2021.

\bibitem[Bau21b]{BaurIsfahan}
Karin Baur.
\newblock Grassmannians and cluster structures.
\newblock {\em Bull. Iranian Math. Soc.}, 47(suppl. 1):S5--S33, 2021.

\bibitem[BBG{\etalchar{+}}24]{BBGTY}
Karin Baur, L{\'e}a Bittmann, Emily Gunawan, Gordana Todorov, and Emine
  Y{\i}ld{\i}r{\i}m.
\newblock {Infinite friezes of affine type D}, 2024.
\newblock \url{https://arxiv.org/pdf/2407.11232}.

\bibitem[BBGE20]{BBGE18}
Karin Baur, Dusko Bogdanic, and Ana Garcia~Elsener.
\newblock Cluster categories from {G}rassmannians and root combinatorics.
\newblock {\em Nagoya Math. J.}, 240:322--354, 2020.

\bibitem[BCJ{\etalchar{+}}24]{BCJKT}
Karin Baur, Ilke Canakci, Karin~M. Jacobsen, Maitreyee~C. Kulkarni, and Gordana
  Todorov.
\newblock Infinite friezes and triangulations of annuli.
\newblock {\em J. Algebra Appl.}, 23(12):Paper No. 2450207, 31, 2024.

\bibitem[BFG{\etalchar{+}}18]{BFGST18}
Karin Baur, Eleonore Faber, Sira Gratz, Khrystyna Serhiyenko, and Gordana
  Todorov.
\newblock Mutation of friezes.
\newblock {\em Bull. Sci. Math.}, 142:1--48, 2018.

\bibitem[BFG{\etalchar{+}}21]{BFGST2}
Karin Baur, Eleonore Faber, Sira Gratz, Khrystyna Serhiyenko, and Gordana
  Todorov.
\newblock Friezes satisfying higher {${\rm SL}_k$}-determinants.
\newblock {\em Algebra Number Theory}, 15(1):29--68, 2021.

\bibitem[BFMS23]{BFMS}
Angelica Benito, Eleonore Faber, Hussein Mourtada, and Bernd Schober.
\newblock Classification of singularities of cluster algebras of finite type:
  the case of trivial coefficients.
\newblock {\em Glasg. Math. J.}, 65(1):170--204, 2023.

\bibitem[BFPT19]{BaurFellnerParsonsTschabold}
Karin Baur, Klemens Fellner, Mark~J. Parsons, and Manuela Tschabold.
\newblock Growth behaviour of periodic tame friezes.
\newblock {\em Rev. Mat. Iberoam.}, 35(2):575--606, 2019.

\bibitem[BG18]{BaurGratz}
Karin Baur and Sira Gratz.
\newblock Transfinite mutations in the completed infinity-gon.
\newblock {\em J. Combin. Theory Ser. A}, 155:321--359, 2018.

\bibitem[BGS87]{BGS}
R.-O. Buchweitz, G.-M. Greuel, and F.-O. Schreyer.
\newblock Cohen-{M}acaulay modules on hypersurface singularities. {II}.
\newblock {\em Invent. Math.}, 88(1):165--182, 1987.

\bibitem[BIRS09]{BIRS}
A.~B. Buan, O.~Iyama, I.~Reiten, and J.~Scott.
\newblock Cluster structures for 2-{C}alabi-{Y}au categories and unipotent
  groups.
\newblock {\em Compos. Math.}, 145(4):1035--1079, 2009.

\bibitem[BKM16]{BKM16}
Karin Baur, Alastair~D. King, and Bethany~R. Marsh.
\newblock Dimer models and cluster categories of {G}rassmannians.
\newblock {\em Proc. Lond. Math. Soc. (3)}, 113(2):213--260, 2016.

\bibitem[BM09]{BM09}
Karin Baur and Bethany~R. Marsh.
\newblock Frieze patterns for punctured discs.
\newblock {\em J. Algebraic Combin.}, 30(3):349--379, 2009.

\bibitem[BMR{\etalchar{+}}06]{BMRRT06}
Aslak~Bakke Buan, Bethany Marsh, Markus Reineke, Idun Reiten, and Gordana
  Todorov.
\newblock Tilting theory and cluster combinatorics.
\newblock {\em Adv. Math.}, 204(2):572--618, 2006.

\bibitem[BMR07]{BuanMarshReiten-TAMS}
Aslak~Bakke Buan, Bethany Marsh, and Idun Reiten.
\newblock Cluster-tilted algebras.
\newblock {\em Trans. Amer. Math. Soc.}, 359(1):323--332, 2007.

\bibitem[BPT16]{BPT}
Karin Baur, Mark~J. Parsons, and Manuela Tschabold.
\newblock Infinite friezes.
\newblock {\em European J. Combin.}, 54:220--237, 2016.

\bibitem[BR10]{BergeronReutenauer}
Francois Bergeron and Christophe Reutenauer.
\newblock {$SL_k$}-tilings of the plane.
\newblock {\em Illinois J. Math.}, 54(1):263--300, 2010.

\bibitem[Buc21]{BuchweitzMCM}
Ragnar-Olaf Buchweitz.
\newblock {\em Maximal {C}ohen-{M}acaulay modules and {T}ate cohomology},
  volume 262 of {\em Mathematical Surveys and Monographs}.
\newblock American Mathematical Society, Providence, RI, [2021] \copyright
  2021.
\newblock With appendices and an introduction by Luchezar L. Avramov, Benjamin
  Briggs, Srikanth B. Iyengar and Janina C. Letz.

\bibitem[CC73a]{CoCo1}
John~H. Conway and Harold S.~M. Coxeter.
\newblock Triangulated polygons and frieze patterns.
\newblock {\em Math. Gaz.}, 57(400):87--94, 1973.

\bibitem[CC73b]{CoCo2}
John~H. Conway and Harold~S.M. Coxeter.
\newblock Triangulated polygons and frieze patterns.
\newblock {\em Math. Gaz.}, 57(401):175--183, 1973.

\bibitem[CC06]{CalderoChapoton}
Philippe Caldero and Fr{\'e}d{\'e}ric Chapoton.
\newblock Cluster algebras as {H}all algebras of quiver representations.
\newblock {\em Comment. Math. Helv.}, 81(3):595--616, 2006.

\bibitem[CCS06]{CCS06}
P.~Caldero, F.~Chapoton, and R.~Schiffler.
\newblock Quivers with relations arising from clusters ({$A_n$} case).
\newblock {\em Trans. Amer. Math. Soc.}, 358(3):1347--1364, 2006.

\bibitem[CF19]{CanakciFelikson}
Ilke Canakci and Anna Felikson.
\newblock Infinite rank surface cluster algebras.
\newblock {\em Adv. Math.}, 352:862--942, 2019.

\bibitem[CFGES25]{CFGS}
Ilke Canakc\i, Francesca Fedele, Ana Garcia~Elsener, and Khrystyna Serhiyenko.
\newblock Super {C}aldero-{C}hapoton map for type {$A$}.
\newblock {\em J. Algebra}, 678:326--375, 2025.

\bibitem[CHJ20]{CuntzHolmJorgensen}
Michael Cuntz, Thorsten Holm, and Peter J{\o}rgensen.
\newblock Frieze patterns with coefficients.
\newblock {\em Forum Math. Sigma}, 8:Paper No. e17, 36, 2020.

\bibitem[CHJ25]{CHJ25}
Michael Cuntz, Thorsten Holm, and Peter J{\o}rgensen.
\newblock Non-commutative friezes and their determinants, the non-commutative
  {L}aurent phenomenon for weak friezes, and frieze gluing.
\newblock {\em Adv. in Appl. Math.}, 171:Paper No. 102940, 33, 2025.

\bibitem[CKP25]{CKP}
Ilke Canakci, Martin Kalck, and Matthew Pressland.
\newblock Cluster categories for completed infinity-gons {I}: {C}ategorifying
  triangulations.
\newblock {\em J. Lond. Math. Soc. (2)}, 111(2):Paper No. e70092, 31, 2025.

\bibitem[Cox69]{Coxeter-Geometry}
H.~S.~M. Coxeter.
\newblock {\em Introduction to geometry}.
\newblock John Wiley \& Sons, Inc., New York-London-Sydney, second edition,
  1969.

\bibitem[Cox71]{Coxeter}
Harold S.~M. Coxeter.
\newblock Frieze patterns.
\newblock {\em Acta Arith.}, 18:297--310, 1971.

\bibitem[Cox91]{Coxeterbook}
H.~S.~M. Coxeter.
\newblock {\em Regular complex polytopes}.
\newblock Cambridge University Press, Cambridge, second edition, 1991.

\bibitem[CR72]{CordesRoselle}
Craig~M. Cordes and D.~P. Roselle.
\newblock Generalized frieze patterns.
\newblock {\em Duke Math. J.}, 39:637--648, 1972.

\bibitem[CR94]{coxeterrigby}
H.~S.~M. Coxeter and J.~F. Rigby.
\newblock Frieze patterns, triangulated polygons and dichromatic symmetry.
\newblock In {\em The Lighter Side of Mathematics}, pages 15--27. The
  Mathematical Association of America, 1994.

\bibitem[Cun17]{Cuntz}
Michael Cuntz.
\newblock On wild frieze patterns.
\newblock {\em Exp. Math.}, 26(3):342--348, 2017.

\bibitem[dSGHL23]{dSGHL}
Antoine de~Saint~Germain, Min Huang, and Jiang-Hua Lu.
\newblock {Friezes of cluster algebras of geometric type}, 2023.
\newblock \url{https://arxiv.org/abs/2309.00906}.

\bibitem[EF25]{EsentepeFaber}
{\"O}zg{\"u}r Esentepe and Eleonore Faber.
\newblock {Penrose tilings, friezes, and the $A_\infty$-singularity}, 2025.
\newblock \url{https://arxiv.org/abs/2511.07530}.

\bibitem[Eis80]{EisenbudMF}
David Eisenbud.
\newblock Homological algebra on a complete intersection, with an application
  to group representations.
\newblock {\em Trans. Amer. Math. Soc.}, 260(1):35--64, 1980.

\bibitem[Fel]{Felikson-frieze-portal}
Anna Felikson.
\newblock {Friezes - Some (mostly online) resources}.
\newblock
  \url{https://www.maths.dur.ac.uk/users/anna.felikson/Projects/frieze/frieze-res.html}.

\bibitem[Fis16]{F16}
Thomas~A. Fisher.
\newblock {On the Enlargement by Pr\"ufer Objects of the Cluster Category of
  type $A_\infty$}, 2016.

\bibitem[FK10]{FuKeller}
Changjian Fu and Bernhard Keller.
\newblock On cluster algebras with coefficients and 2-{C}alabi-{Y}au
  categories.
\newblock {\em Trans. Amer. Math. Soc.}, 362(2):859--895, 2010.

\bibitem[FMP23]{FMP}
Eleonore Faber, Bethany Marsh, and Matthew Pressland.
\newblock {Reduction of Frobenius extriangulated categories}, 2023.
\newblock \url{https://arxiv.org/abs/2308.16232}.

\bibitem[Fon14]{Fontaine}
Bruce Fontaine.
\newblock {Non-zero integral friezes}.
\newblock 2014.
\newblock \url{https://arxiv.org/abs/1409.6026}.

\bibitem[FP16]{FP16}
Bruce Fontaine and Pierre-Guy Plamondon.
\newblock Counting friezes in type {$D_n$}.
\newblock {\em J. Algebraic Combin.}, 44(2):433--445, 2016.

\bibitem[FS21]{FominSetiabrata}
Sergey Fomin and Linus Setiabrata.
\newblock Heronian friezes.
\newblock {\em Int. Math. Res. Not. IMRN}, (1):651--697, 2021.

\bibitem[FS24]{FSLotus}
Eleonore Faber and Bernd Schober.
\newblock {Frieze patterns and combinatorics of curve singularities}.
\newblock 2024.
\newblock \url{https://arxiv.org/abs/2412.02422}.

\bibitem[FWZ16]{FWZ2016}
Sergey Fomin, Lauren Williams, and Andrei Zelevinsky.
\newblock Introduction to cluster algebras. chapters 1-3, 2016.

\bibitem[FWZ17]{FWZ2017}
Sergey Fomin, Lauren Williams, and Andrei Zelevinsky.
\newblock Introduction to cluster algebras. chapters 4-5, 2017.

\bibitem[FWZ20]{FWZ2020}
Sergey Fomin, Lauren Williams, and Andrei Zelevinsky.
\newblock Introduction to cluster algebras. chapters 6, 2020.

\bibitem[FZ02]{FZ2002}
Sergey Fomin and Andrei Zelevinsky.
\newblock Cluster algebras. {I}. {F}oundations.
\newblock {\em J. Amer. Math. Soc.}, 15(2):497--529, 2002.

\bibitem[FZ03]{FZ2003II}
Sergey Fomin and Andrei Zelevinsky.
\newblock Cluster algebras. {II}. {F}inite type classification.
\newblock {\em Invent. Math.}, 154(1):63--121, 2003.

\bibitem[FZ07]{FZ2007}
Sergey Fomin and Andrei Zelevinsky.
\newblock Cluster algebras. {IV}. {C}oefficients.
\newblock {\em Compos. Math.}, 143(1):112--164, 2007.

\bibitem[GG14]{GratzGrabowski}
Jan~E. Grabowski and Sira Gratz.
\newblock Cluster algebras of infinite rank.
\newblock {\em J. Lond. Math. Soc. (2)}, 89(2):337--363, 2014.
\newblock With an appendix by Michael Groechenig.

\bibitem[GK25]{GratzKorff}
Sira Gratz and Christian Korff.
\newblock {Ind-cluster algebras and infinite Grassmannians}, 2025.
\newblock \url{https://arxiv.org/abs/2505.01228}.

\bibitem[GLS08]{GLS-flag}
Christof Geiss, Bernard Leclerc, and Jan Schr\"{o}er.
\newblock Partial flag varieties and preprojective algebras.
\newblock {\em Ann. Inst. Fourier (Grenoble)}, 58(3):825--876, 2008.

\bibitem[GM22]{GunawanMuller}
Emily Gunawan and Greg Muller.
\newblock {Superunitary regions of cluster algebras}.
\newblock 2022.
\newblock \url{https://arxiv.org/abs/2208.14521}.

\bibitem[Guo13]{GuoIMRN}
L.~Guo.
\newblock On tropical friezes associated with {D}ynkin diagrams.
\newblock {\em Int. Math. Res. Not. IMRN}, (18):4243--4284, 2013.

\bibitem[Hap87]{Happel87}
Dieter Happel.
\newblock On the derived category of a finite-dimensional algebra.
\newblock {\em Comment. Math. Helv.}, 62(3):339--389, 1987.

\bibitem[Hap88]{Happel88}
Dieter Happel.
\newblock {\em Triangulated categories in the representation theory of
  finite-dimensional algebras}, volume 119 of {\em London Mathematical Society
  Lecture Note Series}.
\newblock Cambridge University Press, Cambridge, 1988.

\bibitem[Hen13]{Soizic-Henry}
Claire-Soizic Henry.
\newblock Coxeter friezes and triangulations of polygons.
\newblock {\em Amer. Math. Monthly}, 120(6):553--558, 2013.

\bibitem[HJ12]{HJ-infinite}
Thorsten Holm and Peter J{\o}rgensen.
\newblock On a cluster category of infinite {D}ynkin type, and the relation to
  triangulations of the infinity-gon.
\newblock {\em Math. Z.}, 270(1-2):277--295, 2012.

\bibitem[IT15]{ITcyclic}
Kiyoshi Igusa and Gordana Todorov.
\newblock Cluster categories coming from cyclic posets.
\newblock {\em Comm. Algebra}, 43(10):4367--4402, 2015.

\bibitem[IY08]{IY08}
Osamu Iyama and Yuji Yoshino.
\newblock Mutation in triangulated categories and rigid {C}ohen-{M}acaulay
  modules.
\newblock {\em Invent. Math.}, 172(1):117--168, 2008.

\bibitem[JKS16]{JKS16}
Bernt~Tore Jensen, Alastair~D. King, and Xiuping Su.
\newblock A categorification of {G}rassmannian cluster algebras.
\newblock {\em Proc. Lond. Math. Soc. (3)}, 113(2):185--212, 2016.

\bibitem[Kel10]{Keller-quiver-triangulated}
Bernhard Keller.
\newblock Cluster algebras, quiver representations and triangulated categories.
\newblock In {\em Triangulated categories}, volume 375 of {\em London Math.
  Soc. Lecture Note Ser.}, pages 76--160. Cambridge Univ. Press, Cambridge,
  2010.

\bibitem[Kel11]{Keller-Iran}
B.~Keller.
\newblock Cluster algebras and cluster categories.
\newblock {\em Bull. Iranian Math. Soc.}, 37(2):187--234, 2011.

\bibitem[KS11]{KellerScherotzke}
Bernhard Keller and Sarah Scherotzke.
\newblock Linear recurrence relations for cluster variables of affine quivers.
\newblock {\em Adv. Math.}, 228(3):1842--1862, 2011.

\bibitem[KZ08]{KoenigZhu}
Steffen Koenig and Bin Zhu.
\newblock From triangulated categories to abelian categories: cluster tilting
  in a general framework.
\newblock {\em Math. Z.}, 258(1):143--160, 2008.

\bibitem[Mar13]{Marsh13}
Bethany~R. Marsh.
\newblock {\em Lecture notes on cluster algebras}.
\newblock Zurich Lectures in Advanced Mathematics. European Mathematical
  Society (EMS), Z\"urich, 2013.

\bibitem[MG12]{mg12}
Sophie Morier-Genoud.
\newblock Arithmetics of 2-friezes.
\newblock {\em J. Algebraic Combin.}, 36(4):515--539, 2012.

\bibitem[MG15]{Crossroads}
Sophie Morier-Genoud.
\newblock Coxeter's frieze patterns at the crossroads of algebra, geometry and
  combinatorics.
\newblock {\em Bull. Lond. Math. Soc.}, 47(6):895--938, 2015.

\bibitem[MG19]{MG19}
Sophie Morier-Genoud.
\newblock Symplectic frieze patterns.
\newblock {\em SIGMA Symmetry Integrability Geom. Methods Appl.}, 15:Paper No.
  089, 36, 2019.

\bibitem[MGO19]{FareyBoat}
Sophie Morier-Genoud and Valentin Ovsienko.
\newblock Farey boat: continued fractions and triangulations, modular group and
  polygon dissections.
\newblock {\em Jahresber. Dtsch. Math.-Ver.}, 121(2):91--136, 2019.

\bibitem[MGOST14]{mgost14}
Sophie Morier-Genoud, Valentin Ovsienko, Richard~Evan Schwartz, and Serge
  Tabachnikov.
\newblock Linear difference equations, frieze patterns, and the combinatorial
  {G}ale transform.
\newblock {\em Forum Math. Sigma}, 2:e22, 45, 2014.

\bibitem[MGOT12]{MGOT12}
Sophie Morier-Genoud, Valentin Ovsienko, and Serge Tabachnikov.
\newblock 2-frieze patterns and the cluster structure of the space of polygons.
\newblock {\em Ann. Inst. Fourier (Grenoble)}, 62(3):937--987, 2012.

\bibitem[MRZ22]{MohammadiRockZaffalon}
Fatemeh Mohammadi, Job~Daisie Rock, and Francesca Zaffalon.
\newblock {Progress on Infinite Cluster Categories Related to Triangulations of
  the (Punctured) Disk}, 2022.
\newblock \url{https://arxiv.org/abs/2209.15513}.

\bibitem[Mui60]{Muir}
Thomas Muir.
\newblock {\em A treatise on the theory of determinants}.
\newblock Revised and enlarged by William H. Metzler. Dover Publications, Inc.,
  New York, 1960.

\bibitem[Mul23]{Muller-finiteD}
Greg Muller.
\newblock {A short proof of the finiteness of Dynkin friezes}.
\newblock 2023.
\newblock \url{https://arxiv.org/abs/2312.11394}.

\bibitem[Nak23]{NakanishiBook}
Tomoki Nakanishi.
\newblock {\em Cluster algebras and scattering diagrams}, volume~41 of {\em MSJ
  Memoirs}.
\newblock Mathematical Society of Japan, Tokyo, 2023.

\bibitem[NP19]{NakaokaPalu}
Hiroyuki Nakaoka and Yann Palu.
\newblock Extriangulated categories, {H}ovey twin cotorsion pairs and model
  structures.
\newblock {\em Cah. Topol. G\'{e}om. Diff\'{e}r. Cat\'{e}g.}, 60(2):117--193,
  2019.

\bibitem[Pal08]{Palu}
Yann Palu.
\newblock Cluster characters for 2-{C}alabi-{Y}au triangulated categories.
\newblock {\em Ann. Inst. Fourier (Grenoble)}, 58(6):2221--2248, 2008.

\bibitem[Pla18]{Plamondon-CIMPA}
Pierre-Guy Plamondon.
\newblock Cluster characters.
\newblock In {\em Homological methods, representation theory, and cluster
  algebras}, CRM Short Courses, pages 101--125. Springer, Cham, 2018.

\bibitem[Pre23]{Pressland-survey}
Matthew Pressland.
\newblock From frieze patterns to cluster categories.
\newblock In {\em Modern trends in algebra and representation theory}, volume
  486 of {\em London Math. Soc. Lecture Note Ser.}, pages 109--145. Cambridge
  Univ. Press, Cambridge, 2023.

\bibitem[Pro20]{Propp}
James Propp.
\newblock The combinatorics of frieze patterns and {M}arkoff numbers.
\newblock {\em Integers}, 20:Paper No. A12, 38, 2020.

\bibitem[PS25]{PetersenSerhiyenko}
Zachery Peterson and Khrystyna Serhiyenko.
\newblock {${\rm SL}_k$}-{T}ilings and {P}aths in {$\Bbb Z^k$}.
\newblock {\em Int. Math. Res. Not. IMRN}, (21):rnaf331, 2025.

\bibitem[PY21]{PY21}
Charles Paquette and Emine Y{\i}ld{\i}r{\i}m.
\newblock {Completions of discrete cluster categories of type A}.
\newblock {\em Transactions of the London Mathematical Society}, 8(1):35--64,
  Jan 2021.

\bibitem[Rin12]{RingelTropical}
Claus~Michael Ringel.
\newblock Cluster-additive functions on stable translation quivers.
\newblock {\em J. Algebraic Combin.}, 36(3):475--500, 2012.

\bibitem[Sch14]{SchifflerBook}
Ralf Schiffler.
\newblock {\em Quiver representations}.
\newblock CMS Books in Mathematics/Ouvrages de Math\'ematiques de la SMC.
  Springer, Cham, 2014.

\bibitem[Sco06]{Scott06}
Jeanne Scott.
\newblock Grassmannians and cluster algebras.
\newblock {\em Proc. London Math. Soc. (3)}, 92(2):345--380, 2006.

\bibitem[She76]{Shephard}
G.~C. Shephard.
\newblock Additive frieze patterns and multiplication tables.
\newblock {\em Math. Gaz.}, 60(413):178--184, 1976.

\bibitem[Sie83]{Siersma}
Dirk Siersma.
\newblock Isolated line singularities.
\newblock In {\em Singularities, {P}art 2 ({A}rcata, {C}alif., 1981)},
  volume~40 of {\em Proc. Sympos. Pure Math.}, pages 485--496. Amer. Math.
  Soc., Providence, RI, 1983.

\bibitem[Tsc15]{Tschabold}
Manuela Tschabold.
\newblock {Arithmetic infinite friezes from punctured discs}, 2015.
\newblock \url{https://arxiv.org/abs/1503.04352}.

\bibitem[WWZ24]{WWZ2}
Li~Wang, Jiaqun Wei, and Haicheng Zhang.
\newblock Cluster characters for 2-{C}alabi-{Y}au {F}robenius extriangulated
  categories.
\newblock {\em J. Algebra}, 660:645--672, 2024.

\bibitem[Yos90]{Yoshino}
Yuji Yoshino.
\newblock {\em Cohen-{M}acaulay modules over {C}ohen-{M}acaulay rings}, volume
  146 of {\em London Mathematical Society Lecture Note Series}.
\newblock Cambridge University Press, Cambridge, 1990.

\bibitem[Zha25]{Zhang-finite}
Robin Zhang.
\newblock {A Positive Siegel Theorem}.
\newblock 2025.
\newblock \url{https://arxiv.org/abs/2503.08800}.

\end{thebibliography}

\newcommand{\etalchar}[1]{$^{#1}$}
\def\cprime{$'$}

\end{document}